\documentclass[10pt,a4paper, reqno ]{amsart}
\usepackage[foot]{amsaddr}
\usepackage[toc,page]{appendix}
\usepackage[utf8]{inputenc}
\usepackage{amsmath}
\usepackage{amsfonts}
\usepackage{mathtools}
\usepackage{graphicx}
\usepackage{hyperref}
  \usepackage{geometry}
\usepackage{cite}
\usepackage{amsthm}
\usepackage[yyyymmdd,hhmmss]{datetime}
\usepackage{stmaryrd}
\newtheorem{theorem}{Theorem}[section]
\newtheorem{lemma}[theorem]{Lemma}
\newtheorem{remark}[theorem]{Remark}

\newtheorem{definition}[theorem]{Definition}
\usepackage{amssymb}
\newcommand{\R}{\mathbb{R}  }

\newcommand{\pfeil}{ \rightarrow }

\newcommand{\into}{\hookrightarrow}

\renewcommand{\phi}{\varphi}
\renewcommand{\div}{\operatorname{div}}
\newcommand{\ljump}{ \llbracket  }
\newcommand{\rjump}{ \rrbracket }

\title[Two-phase Navier-Stokes/Mullins-Sekerka with boundary contact]{Well-Posedness and qualitative behaviour of a two-phase Navier-Stokes/Mullins-Sekerka system with boundary contact}
\author[Maximilian Rauchecker]{Maximilian Rauchecker}
\address{Maximilian Rauchecker, Fakult\"at f\"ur Mathematik, Universit\"at Regensburg, 93053 Regensburg, Germany}
\author[Mathias Wilke]{Mathias Wilke}
\address{Mathias Wilke,
Institut f\"ur Mathematik, Martin-Luther-Universit\"at Halle-Wittenberg, 06099 Halle, Germany}
\numberwithin{equation}{section} %Numbering style
\begin{document}
\maketitle
\begin{abstract}
We consider a coupled two-phase Navier-Stokes/Mullins-Sekerka system describing the motion of two immiscible, incompressible fluids inside a bounded container. The moving interface separating the liquids meets the boundary of the container at a constant ninety degree angle. This common interface is unknown and has to be determined as a part of the problem.

We show well-posedness and investigate the long-time behaviour of solutions starting close to certain equilibria. We prove that for equal densities these solutions exist globally in time, are stable, and converge to an equilibrium solution at an exponential rate.
\end{abstract}
\section{Introduction}
In this article we study the two-phase Navier-Stokes equations with surface tension coupled to the Mullins-Sekerka problem inside a bounded domain in two or three space dimensions. In our model, the interface separating the two fluids meets the boundary of the domain at a constant ninety degree angle. This leads to a free boundary problem for the interface involving a contact angle problem at the boundary as well.

We assume that the domain $\Omega \subset \R^n$, $n=2,3$, can be decomposed as $\Omega = \Omega^+ (t) \dot\cup \mathring \Gamma (t) \dot\cup \Omega^-(t)$, where $\mathring \Gamma (t)$ 
denotes the interior of $\Gamma(t)$, an $(n-1)$-dimensional submanifold with boundary. We interpret $\Gamma(t)$ to be the interface separating the two phases, $\Omega^+(t)$ and $\Omega^-(t)$, which will both be assumed to be connected. The boundary of $\Gamma(t)$ will be denoted by $\partial\Gamma(t)$. Furthermore we assume $\Gamma(t)$ to be orientable, the unit normal vector field on $\Gamma(t)$ pointing from $\Omega^- (t)$ into $\Omega^+(t)$ will be denoted by $\nu_{\Gamma(t)}$.

Let us introduce some notation. Let $V_{\Gamma(t)}$ denote the normal velocity and $H_{\Gamma(t)}$ the mean curvature of the free interface $\Gamma(t)$. By $\ljump \cdot \rjump$ we denote the jump of a quantity across $\Gamma(t)$ in direction of $\nu_{\Gamma(t)}$, that is,
\begin{equation*}
\ljump f \rjump (x) := \lim_{ \varepsilon \pfeil 0+} [ f(x + \varepsilon \nu_{\Gamma(t)} ) -  f(x - \varepsilon \nu_{\Gamma(t)} ) ], \quad x \in \Gamma(t).
\end{equation*}
Furthermore, $a \otimes b$ is defined by $[a \otimes b]_{ij} := a_i b_j$ for vectors $a,b \in \R^n$ and $A^\top$ denotes the transposed matrix of $A$.

We assume that $\Omega$ is filled by two immiscible, incompressible fluids with respective constant densities $\rho^\pm>0$ in the two phases.
Their respective constant viscosities are denoted by $\mu^\pm > 0$ and $\sigma > 0$ is a given surface tension constant.
To economize our notation, we let $\rho := \rho^+ \chi_{\Omega^+ (t)} + \rho^- \chi_{\Omega^-(t)}$ and $\mu := \mu^+ \chi_{\Omega^+ (t)} + \mu^- \chi_{\Omega^-(t)}$, where $\chi_M$ is the indicator function of a set $M$.
In our model, $u$ is the velocity of the fluids, $p$ the pressure, $\eta$ the chemical potential and $\Gamma(t)$ the free interface at time $t \geq 0$.

Let us consider the case where the domain is a cylindrical container $\Omega = \Sigma \times (L_1,L_2)$, where $ -\infty < L_1 < 0 < L_2 <\infty$ and $\Sigma \subset \R^2$ is bounded and has smooth boundary. By a standard localization method however we can also show well-posedness for smooth, bounded domains. In a forthcoming paper discussing effects of gravity and Rayleigh-Taylor instability, cf. \cite{rtipaper}, this simpler geometry is useful. We denote the lateral walls of the cylinder $\Omega$ by $S_1 := \partial \Sigma \times (L_1,L_2)$ and bottom and top by $S_2 := \Sigma \times \{ L_1, L_2 \}$. As usual, $\nu_{\partial\Omega}$ denotes the unit normal vector field pointing outwards of $\Omega$ and $\nu_{S_1} =\nu_{\partial\Omega}$ on the walls $S_1$. The projection to the tangent space of $S_1$ is defined by $P_{S_1} := I - \nu_{S_1} \otimes \nu_{S_1}$.

In a cylindrical domain the full problem for two possibly different, constant densities and viscosities reads as
\begin{equation} \label{9348576034786508dfgfd7g08697659765}
\begin{alignedat}{2}
\rho \partial_t  u  - \mu \Delta u + \div[ (\rho u + \ljump \rho \rjump \nabla \eta ) \otimes u ] + \nabla p  &= 0, &&\text{in } \Omega \backslash \Gamma(t), \\
\operatorname{div} u &= 0, &&\text{in } \Omega \backslash \Gamma(t), \\
- \ljump \mu (Du + Du^\top) \rjump \nu_{\Gamma (t)} + \ljump p \rjump \nu_{\Gamma (t)} &= \sigma H_{\Gamma(t)} \nu_{\Gamma(t)}, && \text{on } \Gamma(t), \\
\ljump u \rjump &= 0, && \text{on } \Gamma(t), \\
 V_{\Gamma(t)} -u|_{\Gamma(t)} \cdot \nu_{\Gamma(t)} &= - \ljump \nu_{\Gamma(t)} \cdot \nabla \eta \rjump, \quad && \text{on } \Gamma(t), \\
\nu_{\Gamma(t)} \cdot \nu_{S_1} &= 0, && \text {on } \partial \Gamma (t), \\
\Delta \eta &= 0, &&\text {in } \Omega \backslash \Gamma(t), \\
\eta|_{\Gamma(t)} &= \sigma H_{\Gamma(t)} , && \text {on } \Gamma(t), \\
\nu_{\partial\Omega} \cdot \nabla \eta|_{\partial\Omega} &= 0 , && \text{on } \partial\Omega \backslash\Gamma(t), \\
P_{S_1} \left( \mu (Du + Du^\top) \nu_{S_1} \right) &= 0, && \text{on } S_1 \backslash \partial\Gamma(t), \\
u \cdot \nu_{S_1} &= 0, && \text {on } S_1 \backslash \partial \Gamma(t), \\
u &= 0, && \text{on } S_2, \\
u(0) &= u_0, && \text{on } \Omega \backslash \Gamma(0), \\
\Gamma (0) &= \Gamma_0.
\end{alignedat}
\end{equation}
Here we want to mention that we implicitly impose that
$
\Gamma(t) \subset  \Omega$ and  $\partial\Gamma(t) \subset S_1$ for all $t \geq 0$,
that is, the interface stays inside the domain for positive times and the boundary of the interface is contained in the boundary of the domain as well. This only makes sense from a physical standpoint.

\begin{figure}[ht]
	\centering
  \includegraphics[width=0.33\textwidth]{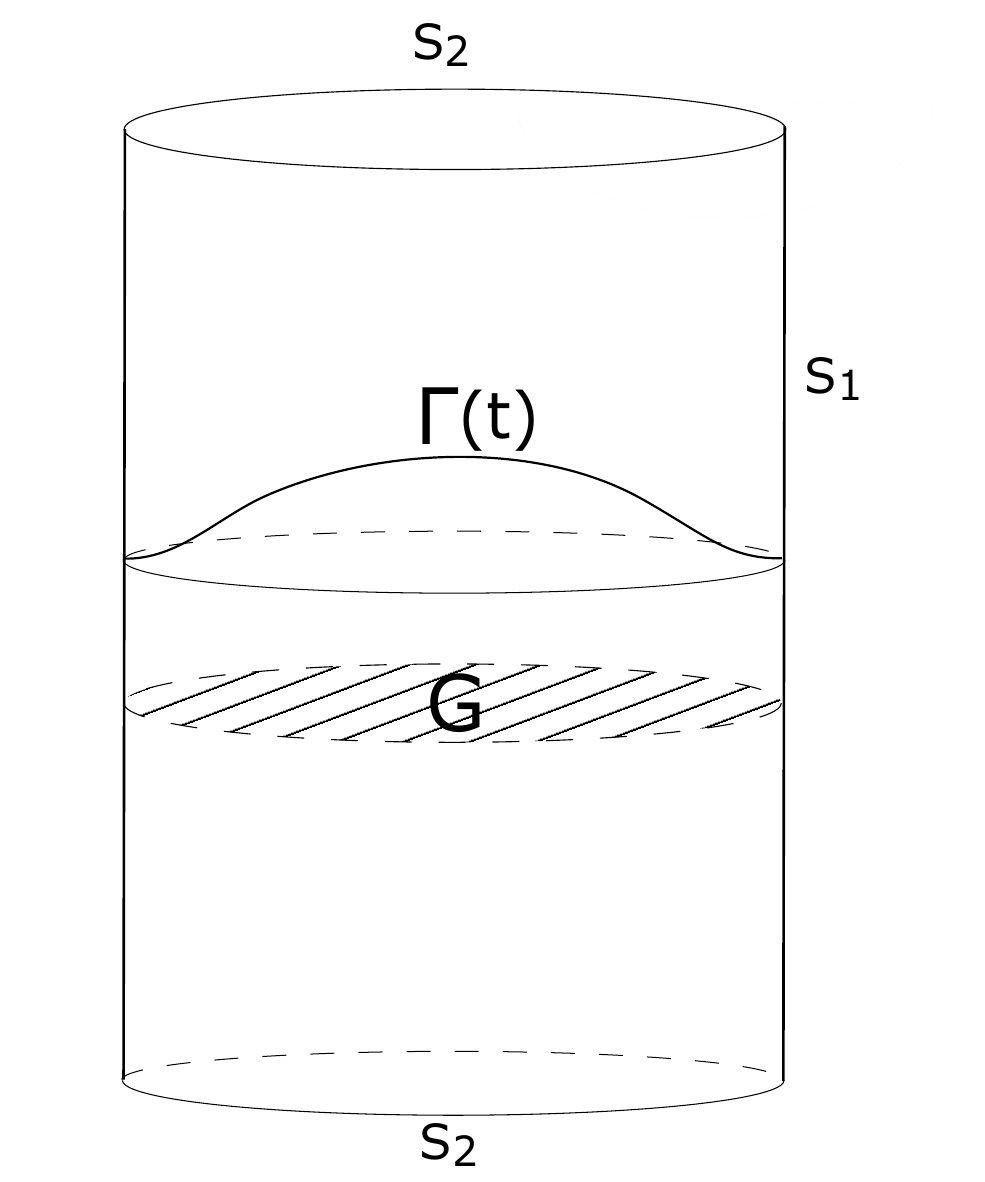}
	\caption{The cylindrical capillary $\Omega$ with lateral walls $S_1$ and bottom and top $S_2$. $G := \Sigma \times \{ 0 \}$ is the reference surface and $\Gamma(t)$ the time-dependent free interface.}
	\label{fig1}
\end{figure}

Note that in this model proposed by Abels, Garcke, and Gr\"un in \cite{aggthermo} the momentum balance $\eqref{9348576034786508dfgfd7g08697659765}_1$ contains an extra term involving the chemical potential $\eta$ since the densities in the two phases are different. This term however is needed to get an energy structure for the system, cf. Section 5 in \cite{aggthermo}. It is shown there that the energy 
\begin{equation}
E(t) := \int_{\Gamma(t)} \sigma d\mathcal H^{n-1} + \frac{1}{2} \int_\Omega \rho(t) u(t)^2 dx
\end{equation}
satisfies the energy-dissipation relation
\begin{equation}
\frac{d}{dt} E(t) = - D(t) := - \int_\Omega \mu  |  \mathbb D u(t)|^2 dx - \int_\Omega | \nabla \eta (t) |^2 dx. 
\end{equation}
Hereby, $\mathbb D u$ is the symmetric part of the gradient $Du$.
There is a remark in order regarding this extra term in $\eqref{9348576034786508dfgfd7g08697659765}_1$. Since $\div u = 0$ and $\Delta \eta = 0$ in the bulk phases $\Omega \backslash \Gamma(t)$, we obtain
\begin{equation*}
 \div[ (\rho u + (\rho^+ - \rho^-) \nabla \eta ) \otimes u ] = \rho (u \cdot \nabla )u + (\rho^+ - \rho^-) (\nabla \eta \cdot \nabla)u, \quad \text{in } \Omega \backslash\Gamma(t).
\end{equation*}
 In the case of equivalent densities, say for simplicity $\rho = 1$, the extra term $\div[  (\rho^+ - \rho^-) \nabla \mu  \otimes u ]$ vanishes and the system reduces to
\begin{equation} \label{9348rtgrt59765}
\begin{alignedat}{2} 
 \partial_t  u  - \mu^\pm \Delta u + (u \cdot \nabla ) u + \nabla p &= 0,& &\text{in } \Omega \backslash \Gamma(t), \\
\operatorname{div} u &= 0, &&\text{in } \Omega \backslash \Gamma(t), \\
- \ljump \mu^\pm (Du + Du^\top) \rjump \nu_{\Gamma (t)} + \ljump p \rjump \nu_{\Gamma (t)} &= \sigma H_{\Gamma(t)} \nu_{\Gamma(t)}, && \text{on } \Gamma(t), \\
\ljump u \rjump &= 0, && \text{on } \Gamma(t), \\
V_{\Gamma(t)} - u|_{\Gamma(t)} \cdot \nu_{\Gamma(t)} &=  - \ljump \nu_{\Gamma(t)} \cdot \nabla \eta \rjump, \quad && \text{on } \Gamma(t), \\
\nu_{\Gamma(t)} \cdot \nu_{S_1} &= 0, && \text {on } \partial \Gamma (t), \\
\Delta \eta &= 0, &&\text {in } \Omega \backslash \Gamma(t), \\
\eta|_{\Gamma(t)} &= \sigma H_{\Gamma(t)} , && \text {on } \Gamma(t), \\
\nu_{\partial\Omega} \cdot \nabla \eta|_{\partial\Omega} &= 0 , && \text{on } \partial\Omega \backslash\Gamma(t), \\
P_{S_1} \left( \mu^\pm (Du + Du^\top) \nu_{S_1} \right) &= 0, && \text{on } S_1 \backslash \partial\Gamma(t), \\
u \cdot \nu_{S_1} &= 0, && \text {on } S_1 \backslash \partial \Gamma(t), \\
u &= 0, && \text{on } S_2, \\
u(0) &= u_0, && \text{on } \Omega \backslash \Gamma(0), \\
\Gamma (0) &= \Gamma_0.
\end{alignedat}
\end{equation}
%Since it is not needed that $\rho^+ \not = \rho^-$, the results of this article also hold true for system \eqref{9348rtgrt59765}. We pose condition \eqref{38495743785} here as well. 
Note that in both cases, \eqref{9348576034786508dfgfd7g08697659765} and \eqref{9348rtgrt59765}, individual masses are conserved,
\begin{equation}
\frac{d}{dt} | \Omega^\pm (t) | =0, \quad t \in \R_+,
\end{equation}
since $\Delta \eta = 0$ and $\div u = 0$ in the bulk phases $\Omega \backslash \Gamma(t)$.

\textbf{Outline of this paper.}
In Section 2 we briefly introduce techniques and functions spaces we work with. In Section 3 we rewrite the free boundary problem of the moving interface as a nonlinear problem for the height function parametrizing the interface. Section 4 deals with an analysis of the underlying linear problem proving maximal regularity in an $L_p-L_q$ scale for the distance function and an $L_r$ scale for the velocity. Section 5 renders that the nonlinear problem is also well-posed, whereas Section 6 deals with qualitative behaviour, stability properties, and convergence to equilibrium solutions.
\section{Preliminaries and function spaces}
We now introduce function spaces and techniques we work with. For a more detailed discussion we refer the reader to the books of Triebel \cite{triebel} and Pr\"uss and Simonett \cite{pruessbuch}.

\subsection{Bessel-Potential, Besov and Triebel-Lizorkin Spaces.}

As usual, we will denote the classical $L_p$-Sobolev spaces on $\R^n$ by $W^k_p(\R^n)$, where $k$ is a natural number and $1 \leq p \leq \infty$. The Bessel-potential spaces will be denoted by $H^s_p(\R^n)$ for $s \in \R$ and the Sobolev-Slobodeckij spaces by $W^s_p(\R^n)$. We will also denote the usual Besov spaces by $B^s_{pr}(\R^n)$, where $s \in \R, 1 \leq p,r \leq \infty$. Lastly, the Triebel-Lizorkin spaces are denoted by $F^s_{pr}(\R^n)$.

These function spaces on a domain $\Omega \subset \R^n$ are defined in a usual way by restriction. The Banach space-valued versions of these spaces are denoted by $L_{p}(\Omega;X)$, $W^{k}_{p}(\Omega;X)$, $H^{s}_{p}(\Omega;X)$, $W^{s}_{p}(\Omega;X)$, $B^{s}_{pr}(\Omega;X)$, $F^{s}_{pr}(\Omega;X)$, respectively. For precise definitions we refer to \cite{meyriesveraarpointwise}.

For results on embeddings, traces, interpolation and extension operators we refer to \cite{abelsbuch}, \cite{danchinbuch}, \cite{lunardioptimal},\cite{lunardiinterpol}, \cite{pruessbuch}, \cite{runst}, \cite{triebel}.

%The following lemma is very well known and can easily be shown by using paraproduct estimates, see \cite{danchinbuch}.
%\begin{lemma} \label{danchinlemma}
%For any $s>0, 1 < p_1,r < \infty$,
%\begin{equation} \label{paraestimate}
%| vw |_{B^s_{p_1r}(\Rn)} \lesssim | v |_{B^s_{p_1r}(\Rn)} |w|_{L_\infty(\Rn)} + |v|_{L_\infty(\Rn)}| w |_{B^s_{p_1r} (\Rn)}
%\end{equation}
%for all $v,w \in B^s_{p_1r} (\Rn) \cap L_\infty(\Rn)$. In particular, the space $B^s_{p_1r} (\Rn) \cap L_\infty(\Rn)$ is an algebra.
%\end{lemma}
%\begin{proof}
%See Corollary 2.86 in \cite{danchinbuch}.
%\end{proof}
\subsection{Maximal Regularity.}
Let us recall the property of maximal $L_p$-regularity, as is e.g. done in Definition 3.5.1 in \cite{pruessbuch}.
\begin{definition}
Let $X$ be a Banach space, $ J = (0,T), 0 < T < \infty$ or $J = \mathbb R_+$, and $A$ a closed, densely defined operator on $X$ with domain $D(A) \subset X.$ Then the operator $A$ is said to have maximal $L_p$-regularity on $J$, if and only if for every $f \in L_p(J;X)$ there is a unique $u \in W^1_p(J;X) \cap L_p(J;D(A))$ solving 
\begin{equation*}
\frac{d}{dt}u (t) + A u (t)= f(t), \quad t \in J, \qquad u|_{t=0} = 0,
\end{equation*}
in an almost-everywhere sense in $L_p(J;X)$.
\end{definition}
There is a wide class of results on operators having maximal regularity, we refer to  \cite{amann}, \cite{amannlineartheory}, \cite{bothe2007lptheoryFA}, \cite{bourgain1984}, \cite{dore2000}, \cite{refpaperpruess}, \cite{prmaxreglpspace}, and \cite{pruessbuch}, for further discussion.

\section{Reduction to a flat interface}
In this section we transform the equations defined on the time-dependent domain $\Omega \backslash \Gamma(t)$ and the moving interface $\Gamma(t)$ to a fixed reference frame. We follow the ideas of \cite{wilkehabil}, see also \cite{hanzawa1981},\cite{prsimNSST}, \cite{mulsekpaper12}, \cite{mulsek2D}, \cite{stefanprobreg}, \cite{rayleighpruess}. To simplify notation let $n = 3$, the modifications for $n = 2$ are obvious.
% and, for sake of brevity, only consider the case of a cylindrical container. 
%However, using similar techniques and curvilinear coordinates in a neighbourhood of a suitable smooth reference surface as in the pure Mullins-Sekerka case \cite{mulsekpaper12} we can extend the strategy to the case of a general bounded and smooth domain. However, the reference surface can be and will be curved in this case so we will need to make a more precise analysis in the case of the bent interface problems. We will give more details later and for simplicity from now on assume that $\Omega$ is a cylindrical container.

 We now assume that the interface at time $t$ is given as a graph over the fixed reference surface $\Sigma := \Omega \cap \{ x_3 = 0 \}.$ More precisely, we assume that there is a height function $h : \Sigma \times \mathbb [0,\infty) \pfeil (L_1,L_2)$, such that
\begin{equation*}
\Gamma (t) = \Gamma_h (t) := \{ x \in \Sigma \times (L_1,L_2) : x_3 = h(x',t), \; x' = (x_1,x_2) \in \Sigma \}, \quad t \geq 0.
\end{equation*}
We will now construct a Hanzawa-type transformation, which is an isomorphism on $\Omega$ and maps the moving interface $\Gamma(t)$ to the reference surface $\Sigma$ for every $t \geq 0$. To this end pick some smooth bump function $\chi \in C^\infty_0(\R; [0,1])$ such that $\chi(s) = 1$ for $|s| \leq \delta/2$ and $\chi(s) = 0$ for $|s| \geq \delta$, where $0 < \delta \leq \min\{ -L_1,L_2 \}/3.$ Define a mapping
\begin{equation*}
\Theta_h : \Omega \times \R_+ \pfeil \Omega, \quad   \Theta_h (x,t) := x + \chi(x_3)h(x',t)e_3 =: x + \theta_h(x,t),
\end{equation*}
where $x = (x',x_3)$.
Then
\begin{equation*}
D \Theta_h =  \begin{pmatrix}
1 & 0 & 0 \\ 0 & 1 & 0 \\
\partial_1 h \chi & \partial_2 h \chi & 1 + h \chi'
\end{pmatrix}.
\end{equation*}
It clearly follows that $D\Theta_h$ is a regular matrix and $\Theta_h$ is invertible, provided $h \chi'$ is sufficiently small. For instance, this is the case whenever
\begin{equation*}
|h|_{L_\infty( (0,T) \times \Sigma )} \leq \frac{1}{2|\chi'|_{L_\infty(\R)}  }.
\end{equation*}
Note that $|\chi'|_\infty$ can be bounded by a constant depending on $\delta$ only.
%Then, given invertibility, one easily computes the inverse to the result
%\begin{equation}
%(D\Theta_h)^{-1} = \frac{1}{1+h\chi'}\begin{pmatrix}
%1+ h\chi'& 0 & 0 \\ 0 & 1+h\chi' & 0 \\
%-\partial_1 h \chi & -\partial_2 h \chi & 1 
%\end{pmatrix}.
%\end{equation}
For the sequel we will fix the bump function $\chi$ and choose $0 < d_0 < 1/(2|\chi'|_\infty)$ sufficiently small and assume that $|h|_{\infty} \leq d_0$. This way we ensure that the inverse $\Theta_h^{-1} : \Omega \pfeil \Omega$ is well defined and maps the free interface $\Gamma(t)$ to the fixed reference surface $\Sigma$.

We will now calculate how the equations behave under this transformation. Define the transformed quantities
\begin{equation*}
w (x,t) := u(\Theta_h(x,t), t) , \quad q (x,t) := p(\Theta_h(x,t), t), \quad  \vartheta(x,t) := \eta( \Theta_h(x,t), t),  
\end{equation*}
for $x \in \Omega,$ $ t \in \mathbb R_+$.
%At this point we note that $w,q$ and $\vartheta$ are defined on the fixed time-independent reference frame. 
We now determine the equations which $(w,q,\vartheta)$ solve. Define $
D\Theta_h^{-\top} := ((D\Theta_h)^{-1} )^\top 
$, as well as the transformed quantities
\begin{equation*}
\nabla_h  := D\Theta_h^{-\top} \nabla, \quad \nabla_h u := (\nabla_h u_k^\top)_{k=1}^3, \quad \div_h := \operatorname{Tr}(\nabla_h), \quad \Delta_h := \div_h \nabla_h.
\end{equation*}
With this it is straightforward to check that
\begin{align*}
\nabla u ( \Theta_h(x,t) , t) = \nabla_h w (x,t), \quad [(u \cdot \nabla) u ](\Theta_h (x,t),t) =[ (w \cdot \nabla_h) w] (x,t),\\ \Delta u (\Theta (x,t),t) = \Delta_h w (x,t), \quad \div u(\Theta(x,t),t) = \div_h w(x,t), \quad x \in \Omega, t \in \R_+.
\end{align*}
Furthermore,
\begin{equation*}
\partial_t u ( \Theta_h (x,t), t ) = \partial_t w (x,t) + Dw(x,t) \partial_t \Theta_h^{-1} ( \Theta_h (x,t) , t), \quad x \in \Omega, t \in \R_+.
\end{equation*}
The upper unit normal at the free interface $\Gamma(t)$ and the normal velocity of which can both be expressed in terms of $h$ by
\begin{equation*} \label{8769856}
\nu_{\Gamma(t)} = \frac{(- \nabla h,1)^\top}{\sqrt{ 1+| \nabla h|^2  }}, \quad V_{\Gamma(t)} = \frac{\partial_t h	 }{  \sqrt{ 1 + | \nabla h |^2 } }, \quad x \in \Sigma, t \in \R_+.
\end{equation*}
We are now able to transform the Two-phase Navier-Stokes/Mullins-Sekerka system \eqref{9348576034786508dfgfd7g08697659765} to the fixed reference frame, the transformed system reads as
\begin{equation} \label{930487547650487gg5}
\begin{alignedat}{2} 
\rho^\pm \partial_t  w  - \mu^\pm \Delta w  + \nabla q &= a^\pm(h; D_x, D_x^2)(w,q) + \bar a(h,w) , \quad &&\text{in } \Omega \backslash \Sigma, \\
\operatorname{div} w &= G_d(h,w), &&\text{in } \Omega \backslash \Sigma, \\
- \ljump \mu^\pm (Dw + Dw^\top) -qI \rjump \nu_{\Sigma }   &=  \sigma \Delta_{x'} h \nu_\Sigma + G_S(h,w,q), && \text{on } \Sigma,  \\
\ljump w \rjump &= 0, && \text{on } \Sigma, \\
\partial_t h &= w \cdot \nu_\Sigma - \ljump \partial_3 \vartheta \rjump + G_\Sigma(h,w,\vartheta), && \text{on } \Sigma, \\
(-\nabla_{x'} h, 1)^\top \cdot \nu_{S_1} &= 0, && \text {on } \partial \Sigma, \\
\Delta \vartheta &= G_c(h,\vartheta), &&\text {in } \Omega \backslash \Sigma, \\
\vartheta|_{\Sigma} - \sigma \Delta_{x'} h &= G_\kappa (h), && \text {on } \Sigma, \\
\nu_{\partial\Omega} \cdot \nabla \vartheta|_{\partial\Omega} &= G_N(h,\vartheta), && \text{on } \partial\Omega \backslash\Sigma, \\
P_{S_1} \left( \mu^\pm (Dw + Dw^\top) \nu_{S_1} \right) &= G_P^\pm (h,w), && \text{on } S_1 \backslash \partial\Sigma, \\
w \cdot \nu_{S_1} &= 0, && \text {on } S_1 \backslash \partial \Sigma, \\
w &= 0, && \text{on } S_2, \\
w(0) &= w_0, && \text{on } \Omega \backslash \Sigma, \\
h (0) &= h_0, && \text{on } \Sigma, 
\end{alignedat}
\end{equation}
where $\nu_\Sigma = e_3$, and 
\begin{equation*}
\begin{alignedat}{1} 
a^\pm(h;D_x,D_x^2)(w,q) &:= \mu^\pm (\Delta_h - \Delta)w + (\nabla - \nabla_h)q, \\
\bar a(h,w,\vartheta) &:= Dw \cdot \partial_t \Theta_h^{-1} - (w \cdot \nabla_h )w  - (\rho^+ - \rho^-)(\nabla_h \vartheta \cdot \nabla_h ) w , \\
G_d(h,w) &:= (\div - \div_h)w, \\
G_S (h,w,q) &:=   \ljump \mu^\pm \left( (D\Theta_h - I)Dw + Dw^\top (D\Theta_h - I)^\top ) \right) \rjump \nu_{\Gamma_h} +    \\
&+ \ljump \left( \mu^\pm (Dw+Dw^\top) - q I \right) (e_3 - \nu_{\Gamma_h} ) \rjump + \sigma( K(h) \nu_{\Gamma_h} - \Delta_{x'} h e_3 ), \\
G_\Sigma (h,w, \vartheta) &:=  w \cdot (-\nabla_{x'} h,0)^\top - \ljump e_3\cdot (\nabla - \nabla_h) \vartheta \rjump - \ljump (-\nabla_{x'} h,0)^\top \cdot \nabla_h \vartheta \rjump, \\
G_c (h,\vartheta) &:= (\Delta-\Delta_h)\vartheta, \\
G_\kappa (h) &:= \sigma(K(h)-\Delta_{x'} h), \\
G_N(h,\vartheta) &:= \nu_{\partial\Omega} \cdot (\nabla-\nabla_h)\vartheta , \\
G_P^\pm(h,w) &:= P_{S_1} \left( \mu^\pm \left( (D\Theta_h - I)Dw + Dw^\top (D\Theta_h - I)^\top ) \right) \nu_{S_1} \right). 
\end{alignedat}
\end{equation*}
Here, cf. \cite{eschersimonett}, the mean curvature is given in terms of $h$ by
\begin{equation*}
K(h)= H(\Gamma_h) = \div \left( \frac{\nabla h}{\sqrt { 1+ |\nabla h|^2 }} \right), \quad x \in \Sigma, t \in \R_+.
\end{equation*}

%Let us briefly explain how we transformed the equations. For instance, pick the evolution equation for the free interface, $V_{\Sigma(t)} = u|_{\Gamma(t)} \cdot \nu_{\Gamma(t)} - \ljump \nu_{\Gamma(t)} \cdot \nabla \mu \rjump$ on $\Gamma(t)$. By employing \eqref{928369872634}, \eqref{8769856} and \eqref{8769856b}, this reads as
%\begin{equation}
%\frac{\partial_t h}{\sqrt{ 1+|\nabla_{x'} h|^2 }} = w \cdot \frac{(-\nabla_{x'} h,1)^\top}{\sqrt{ 1+|\nabla_{x'} h|^2 }} - \frac{ \ljump (-\nabla_{x'} h,1)^\top \cdot \nabla_h \eta \rjump  }{ \sqrt{ 1+|\nabla_{x'} h|^2 } }, \quad x' \in \Sigma, t \in \R_+,
%\end{equation}
%whence we obtain
%\begin{equation}
%\partial_t h = w \cdot e_3 - \ljump \partial_3 \eta \rjump + w \cdot (-\nabla_{x'} h,0)^\top - \ljump e_3\cdot (\nabla - \nabla_h) \eta \rjump - \ljump (-\nabla_{x'} h,0)^\top \cdot \nabla_h \eta \rjump, \quad \text{on } \Sigma, t \in \R_+.
%\end{equation}
Furthermore, we want to point out that we used the fact that the normal $\nu_{S_1}$ is independent of $x_3$ and that the transformation $\Theta_h$ leaves the Dirichlet-boundary $S_2$ invariant.

Since $\nu_\Sigma = e_3$, one can also easily decompose the stress tensor condition $\eqref{930487547650487gg5}_3$ into tangential and horizontal parts, cf. \cite{wilkehabil}. Then, $\eqref{930487547650487gg5}_3$ reads as
\begin{align*}
- \ljump \mu^\pm \partial_3 (w_1,w_2) \rjump - \ljump \mu^\pm \nabla_{x'} w_3 \rjump &= (G_S(h,w,q))_{1,2}, & \text{on } \Sigma, \\
-2 \ljump \mu^\pm \partial_3 w_3 \rjump + \ljump q \rjump - \sigma \Delta_{x'} h &= (G_S(h,w,q))_{3}, & \text{on } \Sigma. 
\end{align*}
To economize notation, we define $$G^\parallel_S (h,w,q) := (G_S(h,w,q))_{1,2}, \quad G^\perp_S (h,w,q) := (G_S(h,w,q))_3 .$$ Hereby we understand $a_{1,2}$ to be $(a_1,a_2)$ for a given vector $a = (a_1,a_2,a_3) \in \mathbb R^3$.
\section{Maximal regularity of the linear problem}
The main goal of this section is to derive a maximal regularity result for the linearization of \eqref{9348576034786508dfgfd7g08697659765}.
\subsection{Linearization, regularity and compatibility conditions}
In this section we consider the linear part of the Two-phase Navier-Stokes/Mullins-Sekerka system, which reads as
\begin{equation} \label{345345345345hh345h345hh345}
\begin{alignedat} {2} \rho^\pm \partial_t  u  - \mu^\pm \Delta u  + \nabla \pi &= g_1 , &&\text{in } \Omega \backslash \Sigma, \\
\operatorname{div} u &= g_2, &&\text{in } \Omega \backslash \Sigma, \\
- \ljump \mu^\pm \partial_3 (u_1,u_2) \rjump - \ljump \mu^\pm \nabla_{x'} u_3 \rjump &= g_3, && \text{on } \Sigma,   \\
-2 \ljump \mu^\pm \partial_3 u_3 \rjump + \ljump \pi \rjump - \sigma \Delta_{x'} h&= g_4, && \text{on } \Sigma, \\
\ljump u \rjump &= g_5, && \text{on } \Sigma, \\
\partial_t h - (u_3^+ + u_3^-)/2+\ljump \partial_3 \eta \rjump &=  g_6, && \text{on } \Sigma, \\
(-\nabla_{x'} h, 0)^\top \cdot \nu_{S_1} &= g_7, && \text {on } \partial \Sigma, \\
\Delta \eta &= g_8, &&\text {in } \Omega \backslash \Sigma, \\
\eta|_{\Sigma} - \sigma \Delta_{x'} h &= g_9, && \text {on } \Sigma, \\
\nu_{\partial\Omega} \cdot \nabla \eta|_{\partial\Omega} &= g_{10}, && \text{on } \partial\Omega \backslash\Sigma, \\
P_{S_1} \left( \mu^\pm (Du + Du^\top) \nu_{S_1} \right) &=P_{S_1}g_{11}, \qquad\qquad\qquad & &\text{on } S_1 \backslash \partial\Sigma, \\
u \cdot \nu_{S_1} &= g_{12}, && \text {on } S_1 \backslash \partial \Sigma, \\
u &= g_{13}, && \text{on } S_2, \\
u(0) &= u_0, && \text{on } \Omega \backslash \Sigma, \\
h (0) &= h_0, && \text{on } \Sigma. 
\end{alignedat}
\end{equation}
Here we take $(u_3^+ + u_3^-)/2$ instead of the trace of $u$ in equation $\eqref{345345345345hh345h345hh345}_6$ since $u$ is allowed to have a jump across $\Sigma.$ Hereby $u_3^\pm$ denote the directional traces of $u_3$ with respect to $\{ x_3 \gtrless 0 \}$.
\subsection{Regularity of the solution}
The question of function spaces is now a very delicate matter. The main idea already used by Abels and Wilke in the case of no boundary contact \cite{abelswilke} is to treat the Navier-Stokes part of the evolution as lower order compared to the Mullins-Sekerka part. They consider some height function $h$ as given and solve the two-phase Navier-Stokes equations in dependence of $h$ by a function $u = u(h)$. Afterwards plugging in the solution $u(h)$ in the evolution equation for $h$ they obtain a problem only dependent on $h$. If now $u$ is sufficiently more regular as the other terms in the evolution equation for $\partial_t h$, the Navier-Stokes equations can be seen as a lower order perturbation. By choosing the time interval sufficiently small one gets well-posedness also for the coupled system, stemming from the unique solvability of the pure Mullins-Sekerka evolution of $h$.

Let us begin by recalling the maximal regularity class for $h$ of the pure Mullins-Sekerka system with boundary contact, cf. \cite{mulsekpaper12}. For $6 < p < \infty$ and $q \in (5/3,2) \cap (2p/(p+1), 2)$, we obtained a unique local in time strong solution
\begin{equation*}
h \in W^1_p(0,T;W^{1-1/q}_q(\Sigma)) \cap L_p(0,T;W^{4-1/q}_q(\Sigma)),
\end{equation*}
and $\eta \in L_p(0,T;W^2_q(\Omega \backslash \Sigma))$ of the linearized Mullins-Sekerka with boundary contact for some $T > 0$, cf. Theorem 5.1 in \cite{mulsekpaper12}.

Two things are important for the analysis: to be later able to treat the Navier-Stokes part as lower order, we need to know that $u|_\Sigma$ has better time regularity and at least as much space regularity as the other terms in $\eqref{345345345345hh345h345hh345}_6$, namely $L_p(0,T;W^{1-1/q}_q(\Sigma))$. On the other hand, the linearized curvature term $\Delta_{x'} h$ has to be at least of the same regularity as $Du|_\Sigma$, cf. $\eqref{345345345345hh345h345hh345}_4$. By choosing a setting where $u$ is too regular, $\Delta_{x' }h$ fails to be admissible data, and by choosing $u$ not regular enough, $u|_\Sigma$ may not be treated as a lower order perturbation. In the following lines we want to explain a setting of function spaces, in which the coupling is of lower order and $u$ is still regular enough to control the nonlinear terms.

The first possibility is to choose an $L_p-L_p$ ansatz, where $p$ as above is large. The vector field $u$ would then be very regular, hence making the nonlinearities easy to handle since in particular $p > 5$. In this ansatz we search for
\begin{equation*}
u \in W^1_p(0,T;L_p(\Omega)) \cap L_p(0,T;W^2_p(\Omega \backslash \Sigma)),
\end{equation*}
whence by classical theory, $u \in BUC([0,T]; W^{2-2/p}_{p} (\Omega \backslash \Sigma))$. Taking traces yields $u|_\Sigma \in BUC([0,T]; W^{2-3/p}_{p} (\Sigma ))$, and hence it can be seen as a lower order perturbation in $L_p(0,T;W^{1-1/q}_q(\Sigma))$. However,
\begin{equation} \label{283764368}
\Delta_{x'} h \in W^{2/3 - 1/(3q)}_p (0,T;L_q(\Sigma)) \cap L_p(0,T;W^{2-1/q}_q(\Sigma)),
\end{equation}
on the other hand,
\begin{equation*}
Du|_\Sigma \in W^{1/2-1/(2p)}_p(0,T;L_p(\Sigma)) \cap L_p(0;T;W^{1-1/p}_p(\Sigma)).
\end{equation*}
It is now a consequence of Sobolev-type embedding theorems to see that $W^{2-1/q}_q(\Sigma)$ does not embed into $W^{1-1/p}_p(\Sigma)$ in general, due to $5/3 < q < 2$ and $p > 6$. Hence this $L_p-L_p$ ansatz with large $p$ does not work.

Alternatively, one can make an $L_q-L_q$ ansatz, searching for some 
\begin{equation*}
u \in W^1_q(0,T;L_q(\Omega)) \cap L_q(0,T;W^2_q(\Omega \backslash \Sigma)),
\end{equation*}
 where $5/3 < q < 2$. Clearly, the function $u$ possesses way less regularity in this ansatz. It is then easy to check that $\Delta_{x'} h$ is admissible data by comparing the regularity classes of $\Delta_{x'} h$ and $Du|_\Sigma$. 
Also, 
\begin{equation*}
u|_\Sigma \in L_{2q/(2-q)}(0,T;W^{1-1/q}_q(\Sigma)).
\end{equation*}
Note that as $q \pfeil 2$, the time regularity index $2q/(2-q)$ tends to $+\infty$. Hence the Stokes part may be treated as lower order whenever $q <2$ is close to $2$.
However we want to point out that handling the nonlinearities may be more difficult since certain Sobolev embeddings fail since $q < 2$.

By choosing an $L_p-L_q$ approach one may get better regularity for $u$, however if one takes any trace of $u$ on the boundary, for instance in the simplest case of the Dirichlet conditions on top and bottom of the container, one ends up with Triebel-Lizorkin spaces in time. It is well known that the optimal regularity for the trace of a function
\begin{equation*}
u \in W^1_p(0,T;L_q(\Omega)) \cap L_p(0,T;W^2_q(\Omega \backslash \Sigma))
\end{equation*}
on the boundary, e.g. $S_2$, is
\begin{equation*}
u|_{S_2} \in F^{1-1/(2q)}_{pq}(0,T; L_q(S_2)) \cap L_p(0,T;W^{2-1/q}_q (S_2)).
\end{equation*}
It is particularly hard to treat this problem in a mixed $L_p-L_q$ setting, since even in the model problems it is not clear how to generalize for instance the results of Pr\"uss and Simonett in \cite{prsimNSST} regarding the Dirichlet-to-Neumann operator. This operator is well understood in an $L_p-L_p$ setting, however the proof of Proposition 3.3 given in \cite{prsimNSST} is not easily generalizable to a mixed setting where $p \not = q$. The proof heavily relies on real interpolation method and Triebel-Lizorkin spaces do not naturally arise as real interpolation spaces.

The explanations above motivate our introduction of a third integration scale. We will show that for given $q <2$ sufficiently close to $2$ and $6 < p < \infty$ finite but large, there is some exponent $ 3 < r = r(q) < 7/2  < \infty$ 
such that the following is true:
$\Delta_{x'} h$ is admissible data in the Stokes part, and $u|_\Sigma$ is lower order in the evolution equation for $h$.
This $L_r-L_r$ approach with $r>3$ circumvents the problem of Triebel-Lizorkin data spaces in the Stokes part completely and hence makes the problem a lot easier to tackle. Also it allows to make use of known results of Pr\"uss and Simonett in \cite{prsimNSST} and makes the nonlinearities easier to handle in the contraction estimates. We will give the precise choice of $r$ below in Theorem \ref{thmr} and prove the above assertions rigorously. 
%However, we want to point out that by more refined estimates, one may be able to consider initial vector fields $u_0$ with lower regularity. The first starting point of this analysis is Remark \ref{9348750w84736504384}.

%Recall that we now have for $u \in W^1_r(0,T;L_r(\Omega)) \cap L_r(0,T;W^2_r(\Omega \backslash \Sigma))$ that
%\begin{equation}
%Du|_\Sigma \in W^{1/2-1/(2r)}_r(0,T;L_r(\Sigma)) \cap L_r(0;T;W^{1-1/r}_r(\Sigma)),
%\end{equation}
%and $\Delta_{x'} h$ has regularity \eqref{283764368}.
\begin{theorem} \label{thmr} Let $n =3$, that is, $\dim \Sigma = n-1 = 2.$
Let $5/3 < q < 2 $ and $6 < p < \infty$. Furthermore, let $0 < T \leq T_0$ for some fixed $T_0 < \infty.$ 
Let
\begin{equation*}
2 \leq r < \frac{7}{6/q -1}.
\end{equation*}
Then, for any $h \in W^1_p(0,T;W^{1-1/q}_q(\Sigma)) \cap L_p(0,T;W^{4-1/q}_q(\Sigma))$, we have that
\begin{equation*} 
 \Delta_{x'} h \in W^{1/2-1/(2r)}_r(0,T;L_r(\Sigma)) \cap L_r(0;T;W^{1-1/r}_r(\Sigma)).
\end{equation*}
Furthermore, there is some $C = C(T) > 0$, such that
\begin{equation} \label{6087698}
\begin{alignedat}{1}
| \Delta_{x'} h &|_{ W^{1/2-1/(2r)}_r(0,T;L_r(\Sigma)) \cap L_r(0;T;W^{1-1/r}_r(\Sigma)) }  \\ &\leq C(T) |h|_{W^1_p(0,T;W^{1-1/q}_q(\Sigma)) \cap L_p(0,T;W^{4-1/q}_q(\Sigma))}.
\end{alignedat}
\end{equation}
 Furthermore, if $2 > q > 9/5$, we can choose $r$ to satisfy $3 < r < 7/2$. 
 If $3 < r < 7/2$, we have
\begin{equation} \label{4354ffff35}
\begin{alignedat}{1}
W^1_r(0,T&;L_r(\Omega)) \cap L_r(0,T;W^2_r(\Omega \backslash \Sigma)) \into \\
&\into  L_\infty(0,T;L_\infty(\Omega)) \cap L_\infty (0,T; W^1_r(\Omega \backslash \Sigma)) \cap L_r(0,T; W^1_\infty(\Omega \backslash \Sigma)).
\end{alignedat}
\end{equation}
%as well as
%$
%Du \in L_\infty(0,T;L_r(\Omega)) \cap L_r(0,T;L_\infty(\Omega)).
%$
%In particular,
Moreover,
\begin{equation}
\operatorname{tr}_\Sigma :W^1_r(0,T;L_r(\Omega)) \cap L_r(0,T;W^2_r(\Omega \backslash \Sigma)) \pfeil L_\infty(0,T;W^{1-1/q}_q(\Sigma)),
\end{equation}
is bounded provided the trace on $\Sigma$ is well defined, for instance if $\ljump u \rjump = 0$. Otherwise the statement is true for the restrictions on $\Omega^\pm$, that is, $\operatorname{tr}_\Sigma^\pm : u \mapsto u^\pm|_\Sigma$, where $u^\pm := u|_{\Omega^\pm}$.

By restricting to height functions $h$ with initial trace zero, $h(0) = 0$, the embedding constant in \eqref{6087698} can be chosen to be independent of $T$ and only depending on $T_0$. In particular, the embedding does not degenerate and the embedding constant stays bounded as $T \downarrow 0$.

Restricting to vanishing traces at $t= 0$ in \eqref{4354ffff35}, the embedding constant is also independent of $T > 0$.
\end{theorem}  \label{eiruherutert}
\begin{proof}
%[Proof of Theorem \ref{eiruherutert}.]
Let $2 \leq r \leq p$ and $0 < T \leq T_0$. 
Note that due to $p \geq r$, we have that $L_p(0,T) \into L_r(0,T)$. The embedding constant here only depends on $T_0$, which stems from H\"older's inequality,
\begin{equation*}
|f|_{L_r(0,T)} \leq T^{(p-r)/(pr)} |f|_{L_p(0,T)}\leq T_0^{(p-r)/(pr)} |f|_{L_p(0,T)}, \quad f \in L_p(0,T).
\end{equation*}
Now, due to Sobolev's embedding theorem, $W^{2-1/q}_q(\Sigma) \into W^{1-1/r}_r(\Sigma)$, provided that $2-3/q > 1-3/r$, which gives an upper restriction on $r$ reading as
\begin{equation} \label{239786423746}
r < \frac{3q}{3-q},
\end{equation}
cf. \cite{abelsbuch}, \cite{triebel}.
Summing up, $L_p(0,T;W^{2-1/q}_q(\Sigma))\into L_r(0,T;W^{1-1/r}_r(\Sigma))$, provided $r \leq p$ and \eqref{239786423746} holds.

Since we want to use the results of \cite{kaipdiss} on the half line, we now consider some
\begin{equation*}
h \in { W}^1_p (\R_+ ; W^{1-1/q}_q(\Sigma)) \cap L_p(\R_+ ; W^{4-1/q}_q(\Sigma)).
\end{equation*}
 Firstly, using Proposition 5.37 in \cite{kaipdiss} on the half line,
 $
 h \in H^\theta_p ( \R_+ ; W^{1-1/q + 3(1-\theta)}_q(\Sigma))
$,
whenever $\theta \in (0,1)$. It follows that
\begin{equation*}
 \Delta_{x'} h \in H^\theta_p ( \R_+ ; W^{2-1/q-3\theta}_q(\Sigma)), \quad \theta \in (0,1).
\end{equation*}
Let $\epsilon > 0 $ small.
By choosing $\theta := 2/3 - 1/q + 2/(3r) - \epsilon \in (0,1)$, we obtain
\begin{equation*}
 \Delta_{x'} h \in H^{  2/3 - 1/q + 2/(3r) - \epsilon }_p ( \R_+ ;  W^{2/q - 2/r + 3\epsilon}_q (\Sigma)).
\end{equation*}
%This now yields the following fact: given a function $h \in W^1_p(0,T; W^{1-1/q}_q(\Sigma)) \cap L_p(0,T;W^{4-1/q}_q(\Sigma))$ we obtain
%\begin{equation}
% \Delta_{x'} h \in H^{  2/3 - 1/q + 2/(3r) }_p ( 0,T ;  H^{2/q - 2/r}_q (\Sigma)),
%\end{equation}
%by standard arguments. 
By Sobolev embeddings for Besov spaces,
\begin{equation*}
 \Delta_{x'} h \in H^{  2/3 - 1/q + 2/(3r) - \epsilon }_p ( 0,T ;  L_r (\Sigma)),
\end{equation*}
for any small $\epsilon > 0$.
Assume for a moment that 
\begin{equation} \label{3948750397864503746}
2/3 - 1/q + 2/(3r) >  1/2 - 1/(2r).
\end{equation}
Then we may choose $\epsilon > 0$ so small, such that
\begin{equation} \label{394875039786jklhl4503746}
2/3 - 1/q + 2/(3r) - \epsilon >  1/2 - 1/(2r).
\end{equation}
 Then $\Delta_{x'} h \in W^{ 1/2 - 1/(2r) }_p ( 0,T ;  L_r (\Sigma)) \into W^{ 1/2 - 1/(2r) }_r ( 0,T ;  L_r (\Sigma))$.
Inequality \eqref{3948750397864503746} however is equivalent to $r < 7/(6/q-1)$ since $q <2$. Estimate \eqref{6087698} is a direct consequence of these considerations. Furthermore, whenever $h$ has vanishing trace at $t=0$, a standard extension argument allows to see that the estimate does not degenerate as $T \pfeil 0$, that is, $C(T)$ stays bounded as $T \pfeil 0$ since it only depends on $T_0$.

Choosing $q < 2$ close enough to $2$ we may assume that $ r > 3$, since $ 7/(6/q-1) \pfeil 7/2$ as $q \pfeil 2$.
%Since $ 7/(6/q-1) \pfeil 7/2$ as $q \pfeil 2$, we easily see that $r > 3$ is possible if only $q$ is close enough to $2$.
Now let $u \in W^1_r(0,T;L_r(\Omega)) \cap L_r(0,T;W^2_r(\Omega \backslash \Sigma)) $ for $r > 3$.
We may use the embedding
\begin{equation}
W^1_r(0,T;L_r(\Omega)) \cap L_r(0,T;W^2_r(\Omega \backslash \Sigma)) \into BUC([0,T]; W^{2-2/r}_r(\Omega \backslash \Sigma)),
\end{equation}
cf. \cite{amann}, to see that $Du \in BUC([0,T] ; W^{1-2/r}_r(\Omega \backslash \Sigma))$, which then in turn yields $Du \in L_\infty(0,T;  L_r(\Omega))$. It also follows that 
%$Du \in H^{1/2 - 3/(2r) - \varepsilon}_r(0,T;L_\infty(\Omega))$ for any $\varepsilon > 0$, \into L_{2r/(5-r)}(0,T;L_\infty(\Omega))$. 
$Du \in L_r(0,T; L_\infty(\Omega))$.
%. Since $r >3$, it is easy to see that $Du \in L_\infty(0,T;L_r(\Omega)) \cap L_r(0,T;L_\infty(\Omega))$.
Regarding the trace operator, we note that 
\begin{equation}
BUC([0,T]; W^{2-2/r}_r(\Omega \backslash \Sigma)) \into L_\infty(0,T; W^1_q(\Omega \backslash \Sigma)),
\end{equation}
whenever $ r \geq 5q/(q+3)$, which is surely satisfied since $r > 3$ and $q < 2$. The proof is complete.
\end{proof}
\begin{remark} \label{9348750w84736504384}
Let us comment on the regularity of solutions.
\begin{enumerate}
\item Note that we can choose from now on $p \in (6,\infty)$, $q \in (9/5,2) \cap (2p/(p+1), 2)$, and $r \in (3,7/2)$. In particular, the set of admissible indices is not empty.
\item Note that if $r > 5/2$, it holds that $u \in BUC([0,T] ; C^0(\Omega))$.
\item There is still room for improvement in these embeddings. For instance, $u$ is $L_\infty(0,T;L_\infty(\Omega))$ whenever $r \geq 5/2$. Furthermore, it can be shown that
\begin{equation*}
Du \in L_{2r/(5-r) - \varepsilon}(0,T;L_\infty(\Omega)) \cap L_\infty(0,T; L_{3r/(5-r)   - \varepsilon}(\Omega)),
\end{equation*}
for any small $\varepsilon > 0$.
This may be used to lower the index $r$ and consider initial data with lower regularity.
\end{enumerate}
\end{remark}

This motivates to choose the following setting for the solutions to the Two-phase Navier-Stokes/Mullins-Sekerka system and its linearization \eqref{345345345345hh345h345hh345}.

Let $T \in (0,\infty)$, $p \in (6,\infty)$, $q \in (9/5,2) \cap (2p/(p+1), 2)$, and $r \in (3,7/2)$ as in Theorem \ref{eiruherutert}. From now on, we will fix the integration scales $p,q$ and $ r$. We are looking for solutions $(u,\pi,h,\mu)$ of  \eqref{345345345345hh345h345hh345} with
\begin{equation*}
\begin{gathered} \label{927364873}
u \in W^1_r(0,T;L_r(\Omega)) \cap L_r(0,T;W^2_r(\Omega \backslash \Sigma)), \quad \pi \in L_r(0,T;\dot H^1_r(\Omega)), \\
\ljump \pi \rjump \in W^{1/2 - 1/(2r)}_r(0,T;L_r(\Sigma)) \cap L_r(0,T;W^{1-1/r}_r(\Sigma)), \\
h \in W^1_p(0,T;W^{1-1/q}_q(\Sigma)) \cap L_p(0,T;W^{4-1/q}_q(\Sigma)), \quad \mu \in L_p(0,T;W^2_q(\Omega \backslash \Sigma)). 
\end{gathered}
\end{equation*}
%We point out that as in \cite{wilkehabil} the additional regularity of $\ljump \pi \rjump$ is determined by the regularity of the Neumann trace of $u$ on $\Sigma$.

\subsection{Regularity of the data.}
To be able to derive a maximal regularity result, we will now deduce optimal regularity classes for the data in problem \eqref{345345345345hh345h345hh345}. Given a solution $(u,\pi, \ljump \pi \rjump, h, \mu)$ in the classes of \eqref{927364873}, we derive by standard trace theory
%\begin{gather}
%\partial_t u, \; \Delta u, \; \nabla \pi \in L_r(0,T;L_r(\Omega)), \quad \Delta \mu \in L_p(0,T;L_q(\Omega)), \\
%u|_\Sigma \in W^{1-1/(2r)}_r(0,T;L_r(\Sigma)) \cap L_r( 0,T; W^{2-1/r}_r(\Sigma)), \\
%u|_{S_j} \in W^{1-1/(2r)}_r(0,T;L_r(S_j)) \cap L_r( 0,T; W^{2-1/r}_r(S_j)), \quad j=1,2, \\
%Du \in W^{1/2}_r (0,T; L_r(\Omega)) \cap L_r(0,T;W^1_r(\Omega \backslash \Sigma)), \label{928389723} \\
%Du|_\Sigma \in W^{1/2-1/(2r)}_r(0,T;L_r(\Sigma)) \cap L_r(0;T;W^{1-1/r}_r(\Sigma)), \\
%Du|_{S_1} \in W^{1/2-1/(2r)}_r(0,T;L_r(S_1)) \cap L_r(0;T;W^{1-1/r}_r(S_1)), \\
%\mu|_\Sigma \in  L_p(0,T;W^{2-1/q}_q(\Sigma)) ,\\
%\nabla \mu|_{\partial\Omega} \in L_p(0,T; W^{1-1/q}_q(\partial\Omega)) , \\
%\partial_t h, \; \ljump \partial_3 \mu \rjump \in L_p(0,T;W^{1-1/q}_q(\Sigma)), \\
%\nabla_{x'} h|_{\partial\Sigma} \in F^{1-2/(3q)}_{pq}(0,T; L_q(\partial\Sigma)) \cap L_p(0,T; W^{3-2/q}_q(\partial\Sigma)) , \label{2345345345jjj345} \\
%u|_{t=0} \in W^{2-2/r}_{r} (\Omega \backslash \Sigma) , \quad h|_{t=0} \in B^{4-3/p-1/q}_{qp}(\Sigma) .
%\end{gather}
%For \eqref{2345345345jjj345} we shall refer to Appendix B in \cite{mulsekpaper12}.
the following necessary conditions for the data,
\begin{equation} \label{9238740374} \begin{gathered}
g_1 \in  L_r(0,T;L_r(\Omega)),  \quad
 g_2 \in L_r(0,T;W^1_r(\Omega \backslash \Sigma)), \\ 
g_3 , \; g_4 \in  W^{1/2-1/(2r)}_r(0,T;L_r(\Sigma)) \cap L_r(0;T;W^{1-1/r}_r(\Sigma)), \\
g_5 \in  W^{1-1/(2r)}_r(0,T;L_r(\Sigma)) \cap L_r( 0,T; W^{2-1/r}_r(\Sigma)),  \\
g_6 \in L_p(0,T;W^{1-1/q}_q(\Sigma)), \quad g_8 \in L_p(0,T;L_q(\Omega)), \\
g_7 \in F^{1-2/(3q)}_{pq}(0,T; L_q(\partial\Sigma)) \cap L_p(0,T; W^{3-2/q}_q(\partial\Sigma)) ,\\
g_9 \in  L_p(0,T;W^{2-1/q}_q(\Sigma)), \quad
g_{10} \in L_p(0,T; W^{1-1/q}_q(\partial\Omega)), \\
P_{S_1}g_{11} \in  W^{1/2-1/(2r)}_r(0,T;L_r(S_1)) \cap L_r(0;T;W^{1-1/r}_r(S_1)),\\
g_{12} \in W^{1-1/(2r)}_r(0,T;L_r(S_1)) \cap L_r( 0,T; W^{2-1/r}_r(S_1)),\\
g_{13} \in W^{1-1/(2r)}_r(0,T;L_r(S_2)) \cap L_r( 0,T; W^{2-1/r}_r(S_2)), \\
u_0 \in W^{2-2/r}_{r} (\Omega \backslash \Sigma) , \quad h_0 \in B^{4-3/p-1/q}_{qp}(\Sigma) . \end{gathered}
\end{equation}
For the regularity of $g_7$ we refer to Appendix A in \cite{mulsekpaper12}.
At this point we note that in \eqref {9238740374} the function $g_2$ does not have to have the time regularity of $Du$ in $\Omega \backslash \Sigma$. This is due to the fact that there is some compatibility condition hidden in the system stemming from the divergence equation, which inherits a certain time regularity for $(g_2,g_5,g_{12},g_{13})$. This will be discussed in the next section regarding compatibility conditions. However we clearly want to point out that $g_2$ being $ L_r(0,T;W^1_r(\Omega \backslash \Sigma))$ alone is a necessary but not a  sufficient condition.

\subsection{Compatibility conditions.}
We now shall discuss all the compatibility conditions for the data $(g_j)_{j=1}^{13}, u_0,h_0$ of system \eqref{345345345345hh345h345hh345}. In Lemma \ref{89306707676340dd} below we rigorously show these conditions all occur and are well-defined. The following observations have already been made in \cite{mulsekpaper12} and \cite{wilkehabil}.

At the starting point of the evolution at time $t=0$ we have to have that 
\begin{equation} \label{9238740374B}
\begin{gathered}
\div u_0 = g_2|_{t=0}, \quad - \ljump \mu^\pm \partial_3 (u_0)_{1,2} \rjump - \ljump \mu^\pm \nabla_{x'} (u_0)_3 \rjump = g_3|_{t=0}, \\
\ljump u_0 \rjump = g_5|_{t=0}, \quad u_0 \cdot \nu_{S_1} = g_{12}|_{t=0}, \quad u_0|_{S_2} = g_{13}|_{t=0}, \\ (-\nabla_{x'} h_0, 1)^\top \cdot \nu_{S_1} = g_7|_{t=0}, \quad 
P_{S_1} ( \mu^\pm (Du_0 + Du_0^\top)\nu_{S_1}) = P_{S_1} g_{11}|_{t=0},
\end{gathered} \end{equation}
by evaluating the respective equations at time zero. Here, $(u_0)_{1,2}$ denotes the vector in $\R^2$ with the first two entries of $u_0$, similarly $(u_0)_3$ denotes the last entry of $u_0$.

Since $\partial\Sigma \subseteq S_1 \not= \emptyset$ and bottom, top and walls of the container have a common boundary, $\partial S_1 \cap \partial S_2 \not= \emptyset$, there are additional compatibility conditions. Simply by comparing equations we get \begin{equation}  \label{9238740374C}
\begin{alignedat}{3}
& \ljump g_{12} \rjump  = g_5 \cdot \nu_{S_1}, \quad \quad && \text{on } \partial \Sigma, \\
& \ljump (g_{11} \cdot e_3)/{\mu^\pm} - \partial_3 g_{12} \rjump = \partial_{\nu_{S_1}} (g_5 \cdot e_3), &&\text{on } \partial\Sigma, \\
& P_{\partial\Sigma} [ ( D_{x'} \Pi g_5 + (D_{x'}\Pi g_5)^\top )\nu_{\partial\Sigma}] = \ljump P_{\partial\Sigma}\Pi g_{11} /{\mu^\pm} \rjump, \quad\quad && \text{on } \partial\Sigma, \\
&g_3 \cdot (\nu_{S_1})_{1,2} = - \ljump g_{11} \cdot e_3 \rjump, && \text{on } \partial\Sigma,  \\
&g_{13} \cdot \nu_{S_1} = g_{12}, && \text{on } \partial{S_2},  \\
&P_{\partial\Sigma}[  \mu^\pm ( D_{x'} \Pi g_{13} + (D_{x'} \Pi g_{13})^\top)\nu_{\partial\Sigma} ] = P_{\partial\Sigma} \Pi g_{11},& &\text{on } \partial S_2,\\
&\mu^\pm \partial_{\nu_{S_1}}(g_{13} \cdot e_3) + \mu^\pm \partial_3 g_{12} = g_{11} \cdot e_3, &&\text{on } \partial S_2. 
\end{alignedat} \end{equation}
Here, $\Pi v := (v_1,v_2) \in \R^2 $ for $v=(v_1,v_2,v_3) \in \R^3$ and $\nu_{\partial\Sigma} := \Pi \nu_{S_1}$. The projection then is given by $P_{\partial\Sigma} := I - \nu_{\partial\Sigma} \otimes \nu_{\partial\Sigma}$.
For further discussion we refer to \cite{wilkehabil}.
%Let us briefly explain how we get these. The first one follows easily from $\ljump g_{12} \rjump =\ljump u \cdot \nu_{S_1} \rjump =  \ljump u \rjump \cdot \nu_{S_1} = g_5 \cdot \nu_{S_1}$ on $\partial\Sigma$. For the second we note that by testing a vector with $e_3$ we end up in the tangent space of $S_1$, hence
%\begin{equation}
%\mu^\pm (Du+Du^\top)\nu_{S_1} \cdot e_3 = g_{11} \cdot e_3, \quad \text{on } S_1.
%\end{equation}
%This renders the equality
%\begin{equation}
%(g_{11} \cdot e_3)/{\mu^\pm} - \partial_3 g_{12} = (Du+Du^\top)\nu_{S_1} \cdot e_3 - \partial_3(u \cdot \nu_{S_1}) = Du^\top \nu_{S_1} \cdot e_3 = Du e_3 \cdot \nu_{S_1},
%\end{equation}
%at least on $S_1$. By definition of the directional derivative and taking the jump brackets we get \eqref{239846378}.

%The other compatibility conditions follow in a similar fashion by comparing the equations on the areas where two conditions meet (interface -- wall, wall -- top/bottom).
%DRITTE

%Now, \eqref{876067} is simply obtained by combining the pure-slip condition on $S_1$ together with the balance equation for the stress tensor, \eqref{876067b} simply states the compatibility between the Dirichlet condition on top and bottom and Navier condition on the walls for $u$.

We want to point out that there is no additional compatibility condition for $\partial_t g_7$ on $\partial \Sigma$ as there is in \cite{wilkehabil}, since $g_7$ does not have a well defined time derivative on $\partial \Sigma$ in our regularity class. This is due to the fact that we have a different maximal regularity class for $h$ as in \cite{wilkehabil}.

Finally we turn to the divergence equation and want to point out that there is another compatibility and regularity condition hidden in the system, which has already been investigated in \cite{wilkehabil}. For completeness we explain it here briefly.

Consider the divergence equation $\div u = g_2$ and multiply this equation with a testfunction $\phi \in W^1_{r'} (\Omega)$, where $r' = r/(r-1)$ is the conjugate exponent. An integration by parts on the two Lipschitz domains $\Omega \cap \{ x_3 \gtrless 0 \}$ %entails that
%\begin{align}
%\int_{\Omega \backslash \Sigma} \div u \; \phi dx = \int_{S_1} (u \cdot \nu_{S_1}) \phi|_{S_1} d S_1 + \int_{S_2} (u \cdot \nu_{S_2}) \phi|_{S_2} d S_2  \\ - \int_\Sigma ( \ljump u \rjump \cdot \nu_\Sigma ) \phi|_\Sigma d\Sigma + \int_{\Omega \backslash \Sigma} u \cdot \nabla \phi dx,
%\end{align}
and using the equations entails that
%see also Proposition A.14 in \cite{wilkehabil}.
%Employing the equations this reads as
\begin{equation} \begin{split} \label{345345}
\int_{\Omega \backslash \Sigma} g_2 \phi dx - \int_{S_1} g_{12} \phi|_{S_1} d S_1 - \int_{S_2} (g_{13} \cdot \nu_{S_2}) \phi|_{S_2} d S_2  \\ + \int_\Sigma ( g_5 \cdot \nu_\Sigma ) \phi|_\Sigma d\Sigma =- \int_{\Omega \backslash \Sigma} u \cdot \nabla \phi dx, \end{split}
\end{equation}
see also Proposition A.14 in \cite{wilkehabil}.
Hence the functional $\phi \mapsto \langle (g_2,g_5,g_{12},g_{13}), \phi \rangle$ defined by the left hand side of \eqref{345345} is continuous on $W^1_{r'}(\Omega)$ with respect to the seminorm $| \nabla \cdot |_{L_{r'}(\Omega)}$. Since $C_0^\infty(\Omega) \subseteq W^1_{r'}(\Omega)$ is dense in the homogeneous space $\dot H^1_{r'}(\Omega)$ with respect to this seminorm, it follows that $\phi \mapsto \langle (g_2,g_5,g_{12},g_{13}), \phi \rangle$ defines a functional on $\dot H^1_{r'}(\Omega)$. In other words,
$
(g_2,g_5,g_{12},g_{13}) \in \hat H^{-1}_r(\Omega) := (\dot H^1_{r'}(\Omega))'.
$
The norm of $(g_2,g_5,g_{12},g_{13}) $ in $\hat H^{-1}_r(\Omega)$ is then given by
\begin{equation*}
| (g_2,g_5,g_{12},g_{13}) |_{\hat H^{-1}_r(\Omega)} := \sup \{   \langle (g_2,g_5,g_{12},g_{13}), \phi \rangle / | \nabla \phi |_{L_{r'}(\Omega)} : \phi \in W^{1}_{r'}(\Omega)   \}.
\end{equation*}
We now turn again to the equations. Since $u \in W^1_r(0,T;L_r(\Omega))$, it follows from \eqref{345345} that
$
\frac{d}{dt}(g_2,g_5,g_{12},g_{13})
$
is well defined and is in $L_r(0,T; \hat H^{-1}_r(\Omega))$. Consequently,
\begin{equation} \label{9238740374D}
(g_2,g_5,g_{12},g_{13}) \in W^1_r(0,T; \hat H^{-1}_r(\Omega))
\end{equation}
is another necessary compatibility and regularity condition.
We close this subsection by showing that the compatibility conditions we have deduced above are all well-defined conditions.
\begin{lemma}
\label{89306707676340dd} 
Let $ r > 3$. Then all appearing traces and hence the compatibility conditions are all well-defined.
\end{lemma}
\begin{proof}
Firstly, $g_j$, $j=2,3,5,7,12,13$, and $P_{S_1}g_{11}$ all have a well-defined trace at $t=0$ since $r > 3$. Indeed, the condition for $g_7$ is independent of $r$  (and fulfilled by choice of $p$ and $q$) and the rest easily follow by trace theory. Pick for instance $g_3$. Then $g_3$ surely has a trace at $t=0$ whenever $1/2 - 1/(2r) - 1/r > 0$. This is however equivalent to $r > 3$. 
By taking traces in the spatial variables one easily sees that all the other traces are well-defined. 
\end{proof}

\subsection{Maximal regularity}
Let us consider the linear problem
\begin{equation} \label{234345345345kk34kk3k453453}
\begin{alignedat}{2} 
\rho^\pm \partial_t  u  - \mu^\pm \Delta u  + \nabla \pi &= g_1 , &&\text{in } \Omega \backslash \Sigma, \\
\operatorname{div} u &= g_2, &&\text{in } \Omega \backslash \Sigma, \\
- \ljump \mu^\pm \partial_3 (u_1,u_2) \rjump - \ljump \mu^\pm \nabla_{x'} u_3 \rjump &= g_3, && \text{on } \Sigma,   \\
-2 \ljump \mu^\pm \partial_3 u_3 \rjump + \ljump \pi \rjump - \sigma \Delta_{x'} h &= g_4, && \text{on } \Sigma, \\
\ljump u \rjump &= g_5, && \text{on } \Sigma, \\
\partial_t h  - (u_3^+ + u_3^-)/2+\ljump \partial_3 \mu \rjump &=  g_6, && \text{on } \Sigma, \\
(-\nabla_{x'} h, 1)^\top \cdot \nu_{S_1} &= g_7, && \text {on } \partial \Sigma, \\
\Delta \mu &= g_8, &&\text {in } \Omega \backslash \Sigma, \\
\mu|_{\Sigma} - \sigma \Delta_{x'} h &= g_9, && \text {on } \Sigma, \\
\nu_{\partial\Omega} \cdot \nabla \mu|_{\partial\Omega} &= g_{10}, && \text{on } \partial\Omega \backslash\Sigma, \\
P_{S_1} \left( \mu^\pm (Du + Du^\top) \nu_{S_1} \right) &=P_{S_1}g_{11}, \quad\quad\quad && \text{on } S_1 \backslash \partial\Sigma, \\
u \cdot \nu_{S_1} &= g_{12}, && \text {on } S_1 \backslash \partial \Sigma, \\
u &= g_{13}, && \text{on } S_2, \\
u(0) &= u_0, && \text{on } \Omega \backslash \Sigma, \\
h (0) &= h_0, && \text{on } \Sigma. 
\end{alignedat}
\end{equation}
The main result on maximal regularity for \eqref{234345345345kk34kk3k453453} is the following.
\begin{theorem} \label{maxregodfpgudfphgu}
Let $\mu^\pm, \rho^\pm, \sigma >0$ be constant, $-\infty < L_1 < 0 < L_2 < \infty$, $(p,q,r)$ as in Theorem \ref{thmr} and $\Sigma \subset \R^2$ be a bounded, smooth domain. Let $\Omega := \Sigma \times (L_1,L_2)$, $S_1 := \partial\Sigma \times (L_1,L_2)$, and $S_2 := \Sigma \times \{ L_1 , L_2\}$. Let $0 < T < \infty$. The coupled linear system \eqref{234345345345kk34kk3k453453} then admits a unique solution $(u,\pi, \ljump \pi \rjump, h , \mu)$ with regularity \eqref{927364873},
if and only if the data satisfy the regularity and compatibility conditions
\eqref{9238740374}, \eqref{9238740374B}, \eqref{9238740374C}, and \eqref{9238740374D}. 
Furthermore, the solution map $[((g_j)_{j = 1,...,13},u_0,h_0) \mapsto (u, \pi, \ljump \pi \rjump, h, \mu)]$ between the above spaces is continuous.
\end{theorem}
\begin{proof}
First we reduce to trivial initial data by solving an auxiliary ninety degree angle linear Mullins-Sekerka problem of type

\begin{align*}
\partial_t \bar h  +\ljump \partial_3  \bar \mu \rjump &=  g_6, & \text{on } \Sigma, \\
(-\nabla_{x'}  \bar h, 1)^\top \cdot \nu_{S_1} &= g_7, & \text {on } \partial \Sigma, \\
\Delta  \bar \mu &= g_8, &\text {in } \Omega \backslash \Sigma, \\
 \bar \mu|_{\Sigma} - \sigma \Delta_{x'}  \bar h &= g_9, & \text {on } \Sigma, \\
\nu_{\partial\Omega} \cdot \nabla  \bar \mu|_{\partial\Omega} &= g_{10}, & \text{on } \partial\Omega \backslash\Sigma, \\
 \bar h (0) &= h_0, & \text{on } \Sigma, 
\end{align*}

by functions
\begin{equation*}
\bar h \in W^1_p(0,T;W^{1-1/q}_q(\Sigma)) \cap L_p(0,T;W^{4-1/q}_q(\Sigma)), \quad \bar \mu \in L_p(0,T;W^2_q(\Omega \backslash \Sigma)).
\end{equation*}
%Here we point out that the compatibility condition is satisfied, hence the problem is uniquely solvable.
Then we solve an auxiliary two-phase Stokes problem
\begin{equation}\label{hehugrferz6823543}
\begin{alignedat}{2}
\rho^\pm \partial_t \bar  u  - \mu^\pm \Delta  \bar u  + \nabla \bar  \pi &= g_1 , &&\text{in } \Omega \backslash \Sigma, \\
\operatorname{div}  \bar u &= g_2, &&\text{in } \Omega \backslash \Sigma, \\
- \ljump \mu^\pm \partial_3 ( \bar u_1, \bar u_2) \rjump - \ljump \mu^\pm \nabla_{x'}  \bar u_3 \rjump &= g_3, & &\text{on } \Sigma,   \\
-2 \ljump \mu^\pm \partial_3  \bar u_3 \rjump + \ljump  \bar \pi \rjump   &= g_4 - \sigma \Delta_{x'} \bar h, \quad\quad && \text{on } \Sigma,  \\
\ljump  \bar u \rjump &= g_5, & &\text{on } \Sigma, \\
P_{S_1} \left( \mu^\pm (D \bar u + D \bar u^\top) \nu_{S_1} \right) &=P_{S_1}g_{11}, && \text{on } S_1 \backslash \partial\Sigma, \\
 \bar u \cdot \nu_{S_1} &= g_{12}, && \text {on } S_1 \backslash \partial \Sigma, \\
 \bar u &= g_{13}, & &\text{on } S_2, \\
 \bar u(0) &= u_0, && \text{on } \Omega \backslash \Sigma,
\end{alignedat}
\end{equation}
using Theorem A.11 in \cite{wilkehabil} by functions
\begin{equation}
\bar u \in W^1_r(0,T; L_r(\Omega)) \cap L_r(0,T;W^2_r(\Omega \backslash \Sigma)), \quad \bar \pi \in L_r(0,T; \dot H^1_r(\Omega \backslash \Sigma)),
\end{equation}
with $\ljump \bar \pi \rjump \in W^{1/2 - 1/(2r)}_r(0,T;L_r(\Sigma)) \cap L_r(0,T;W^{1-1/r}_r(\Sigma))$. Here we want to point out two things: $\Delta_{x'} \bar h$ has sufficient regularity to be admissible data and that there is no compatibility condition stemming from $\eqref{hehugrferz6823543}_4$. Hence $g_4 - \sigma \Delta_{x'}\bar h$ is admissible data for the problem. Having now $(\bar u, \bar \pi, \bar h, \bar \mu)$ at hand, we are left to solve
\begin{align*} 
\rho^\pm \partial_t  u  - \mu^\pm \Delta u  + \nabla \pi &= 0 , &&\text{in } \Omega \backslash \Sigma, \\
\operatorname{div} u &= 0, &&\text{in } \Omega \backslash \Sigma, \\
- \ljump \mu^\pm \partial_3 (u_1,u_2) \rjump - \ljump \mu^\pm \nabla_{x'} u_3 \rjump &= 0, && \text{on } \Sigma,    \\
-2 \ljump \mu^\pm \partial_3 u_3 \rjump + \ljump \pi \rjump - \sigma \Delta_{x'} h  &= 0, && \text{on } \Sigma, \\
\ljump u \rjump &= 0, && \text{on } \Sigma, \\
\partial_t h  - u_3|_\Sigma +\ljump \partial_3 \mu \rjump &=  (\bar u_3^+ + \bar u_3^-)/2, && \text{on } \Sigma, \\
(-\nabla_{x'} h, 1)^\top \cdot \nu_{S_1} &= 0, && \text {on } \partial \Sigma, \\
\Delta \mu &= 0, &&\text {in } \Omega \backslash \Sigma, \\
\mu|_{\Sigma} - \sigma \Delta_{x'} h &= - \sigma\Delta_{x'} \bar h, && \text {on } \Sigma, \\
\nu_{\partial\Omega} \cdot \nabla \mu|_{\partial\Omega} &= 0, && \text{on } \partial\Omega \backslash\Sigma, \\
P_{S_1} \left( \mu^\pm (Du + Du^\top) \nu_{S_1} \right) &=0, && \text{on } S_1 \backslash \partial\Sigma, \\
u \cdot \nu_{S_1} &= 0,& & \text {on } S_1 \backslash \partial \Sigma, \\
u &= 0, && \text{on } S_2, \\
u(0) &= 0, && \text{on } \Omega \backslash \Sigma, \\
h (0) &= 0, && \text{on } \Sigma.
\end{align*}
We do this as follows. Define $L_{MS} : {_0} \mathbb E_{MS,T} \pfeil {_0} \mathbb F_{MS,T}$ by
\begin{equation*}
L_{MS} : (h,\mu) \mapsto \begin{pmatrix}
\partial_t h - \ljump \partial_3 \mu \rjump \\ \Delta \mu  \\ \mu|_\Sigma - \sigma \Delta_{x'} h \\ n_{\partial\Omega} \cdot \nabla \mu |_{\partial\Omega} \\ (-\nabla_{x'} h|_{\partial\Sigma}, 1)^\top \cdot \nu_{S_1}
\end{pmatrix}
\end{equation*}
where
\begin{equation*}
_0 \mathbb E_{MS,T} := [ {_0 W}^1_p (0,T; W^{1-1/q}_q(\Sigma)) \cap L_p(0,T; W^{4-1/q}_q(\Sigma)) ] \times L_p(0,T; W^2_q(\Omega \backslash \Sigma)),
\end{equation*}
and
\begin{equation*}
\begin{alignedat}{1}
_0 &\mathbb  F_{MS,T} := L_p(0,T, W^{1-1/q}_q(\Sigma)) \times L_p(0,T; L_q(\Omega)) \times L_p(0,T; W^{1-1/q}_q(\Sigma)) \\
&\times L_p(0,T; W^{1-1/q}_q(\partial \Omega)) \times [  {_0 F}^{1-2/(3q)}_{pq}(0,T; L_q(\partial\Sigma)) \cap L_p(0,T; W^{3-2/q}_q(\partial\Sigma))   ].
\end{alignedat}
\end{equation*}
In \cite{mulsekpaper12} we have shown that $L_{MS} : {_0} \mathbb E_{MS,T} \pfeil {_0} \mathbb F_{MS,T}$ is boundedly invertible. Define $L_S : {_0 \mathbb E}_{MS,T} \pfeil [{_0 W}^1_r(0,T;L_r(\Omega)) \cap L_r(0,T;W^2_r(\Omega \backslash \Sigma))]$ by $L_S(h) := u$, where $(u,\pi)$ is the unique solution of
\begin{align*} 
\rho^\pm \partial_t  u  - \mu^\pm \Delta u  + \nabla \pi &= 0 , &&\text{in } \Omega \backslash \Sigma, \\
\operatorname{div} u &= 0, &&\text{in } \Omega \backslash \Sigma, \\
- \ljump \mu^\pm \partial_3 (u_1,u_2) \rjump - \ljump \mu^\pm \nabla_{x'} u_3 \rjump &= 0, && \text{on } \Sigma,    \\
-2 \ljump \mu^\pm \partial_3 u_3 \rjump + \ljump \pi \rjump   &= \sigma \Delta_{x'} h,& & \text{on } \Sigma, \\
\ljump u \rjump &= 0, && \text{on } \Sigma, \\
P_{S_1} \left( \mu^\pm (Du + Du^\top) \nu_{S_1} \right) &=0, && \text{on } S_1 \backslash \partial\Sigma, \\
u \cdot \nu_{S_1} &= 0, && \text {on } S_1 \backslash \partial \Sigma, \\
u &= 0, && \text{on } S_2, \\
u(0) &= 0, && \text{on } \Omega \backslash \Sigma,
\end{align*}
cf. Theorem A.11 in \cite{wilkehabil}.
It then stems from Theorem \ref{thmr} that $L_S$ is well defined, linear and bounded. 
Define $B : {_0 \mathbb E}_{MS,T} \pfeil {_0 \mathbb F}_{MS,T}$ and $G(\bar u , \bar h) \in{_0 \mathbb F}_{MS,T}$ by
\begin{equation*}
B(h) := (-L_S(h)|_\Sigma, 0,0,0,0)^\top, \quad G(\bar u, \bar h) := ( (\bar u_3^+ + \bar u_3^-)/2 ,0, -\sigma \Delta_{x'} \bar h, 0,0)^\top.
\end{equation*}
We can hence rewrite the problem as
\begin{equation*}
L_{MS}(h,\mu) = -B(h) + G(\bar u , \bar h), \quad \text{ in } {_0 \mathbb F}_{MS,T}.
\end{equation*}
We now solve this equation by a Neumann series argument. Clearly this equation is equivalent to
\begin{equation*}
 ( I + L_{MS}^{-1} B) (h,\mu) = L_{MS}^{-1} G(\bar u, \bar h), \quad \text{ in } {_0 \mathbb F}_{MS,T},
\end{equation*}
hence it remains to show that
$
| L_{MS}^{-1} B |_{\mathcal B( {_0 \mathbb E_{MS,T}} )} \leq \frac{1}{2},
$
if $T > 0$ is small enough.
Then by a Neumann series argument, $( I + L_{MS}^{-1} B)$ is invertible and the theorem is shown. Since we now have that $L_{MS}$ is boundedly invertible and the norm of the inverse is independent of $T$ - recall we only consider functions with vanishing time trace at $t=0$ - the claim follows from Theorem \ref{thmr}. Indeed,
\begin{equation*}
\begin{alignedat}{1}
| B(h)|_{ {_0 \mathbb F}_{MS,\tau}} &= | L_S(h)|_{L_p(0,\tau;W^{1-1/q}_q(\Sigma))} \leq \tau^{1/p} |L_S(h)|_{L_\infty(0,\tau;W^{1-1/q}_q(\Sigma))} \\ &\leq \tau^{1/p}|h|_{ _0 \mathbb E_{MS,\tau}}, \quad \tau > 0.
\end{alignedat}
\end{equation*}
Note that again since $h$ has vanishing time trace, all embeddings in Theorem \ref{thmr} are time-independent. In particular, by choosing $\tau>0$ sufficiently small, we get a unique solution $(h,\mu)$ in the proper regularity class on $(0,\tau)$. Solving then the two-phase Stokes system for this particular $h$ gives a proper $(u,\pi)$ in the $L_r$-regularity scale, again on $(0,\tau)$. 

Shifting back the equations via $\tilde u(t) := u(t-\tau), \tilde \pi(t) := \pi(t-\tau), \tilde h(t) := h(t-\tau)$ and $\tilde \mu := \mu(t-\tau)$ we can again apply this argument and solve again on the same length time interval $(0,\tau)$, which in turn gives us now a solution on $(0,2\tau)$ in fact. Repeating the steps we can solve then the problem on $(0,T)$, cf. Section 2.3 in \cite{wilkehabil}.
\end{proof}

\section{Nonlinear Well-Posedness}
In this section we show local well-posedness for the full nonlinear problem \eqref{930487547650487gg5}. The main result is the following.
\begin{theorem}  \label{986876gg2983}
Let $\mu^\pm, \rho^\pm, \sigma >0$ be constant, $-\infty < L_1 < 0 < L_2 < \infty$, $p \in (6,\infty)$, $q \in (2p/(p+1) , 2) \cap (19/10 , 2)$, $ 3 < r <7/2$, $\Sigma \subset \R^2$ be a bounded, smooth domain. Let $\Omega := \Sigma \times (L_1,L_2)$, $S_1 := \partial\Sigma \times (L_1,L_2)$ be the walls and $S_2 := \Sigma \times \{ L_1 , L_2\}$ bottom and top of the container.
Furthermore let
$
(u_0 , h_0)\in W^{2-2/r}_r(\Omega \backslash \Sigma) \times  B^{4-1/q-3/p}_{qp}(\Sigma)$
be admissible by the compatibility conditions
\begin{equation}\label{937846586734b}
\begin{gathered} 
\div u_0 = G_d(h_0,u_0) , \quad \text{in } \Omega \backslash \Sigma, \\
- \ljump \mu^\pm \partial_3 (u_0)_{1,2} \rjump - \ljump \mu^\pm \nabla_{x'} (u_0)_3 \rjump = G_S^\parallel (u_0,h_0), \;  
\ljump u_0 \rjump = 0, \quad\text{on } \Sigma,\\
P_{S_1} ( \mu^\pm (Du_0 + Du_0^\top)\nu_{S_1}) = 0, \;
u_0 \cdot \nu_{S_1} = 0, \quad\text{on } S_1, \\
u_0|_{S_2} = 0, \; \text{on } S_2, \quad
(-\nabla_{x'} h_0, 1)^\top \cdot \nu_{S_1} = 0, \; \text{on } \partial \Sigma. 
\end{gathered}
\end{equation}

Then the full nonlinear (transformed) problem \eqref{930487547650487gg5} admits a unique local-in-time strong solution, that is, there is some $T_0 > 0$, such that for every $0 < T \leq T_0$ there is some $\varepsilon = \varepsilon(T) > 0$, such that whenever the smallness condition
\begin{equation} \label{98236736408763046ffdfg} 
| u_0 |_{W^{2-2/r}_{r}(\Omega\backslash \Sigma)} +
 | h_0 |_{B^{4-1/q-3/p}_{qp}(\Sigma)} \leq \varepsilon
\end{equation}
is satisfied there is a unique strong solution $(u,\pi, \ljump \pi \rjump, h, \mu)$
of \eqref{930487547650487gg5} on $(0,T)$ with regularity \eqref{927364873}.
\end{theorem}
\begin{proof}
We first again reduce the problem to $(u_0,h_0) = 0$. This can be done by solving an auxiliary problem first
%Due to the compatibility condition $\eqref{937846586734b}_7$ we can uniquely solve the linear problem
%\begin{subequations}
%\begin{align} \label{eq12adfgdfgdfgdfgertertTRL}
%\rho^\pm \partial_t  u_* - \mu^\pm \Delta u_*  + \nabla \pi_* &= 0 , &&\text{in } \Omega \backslash \Sigma, \\
%\operatorname{div} u_* &= 0, &&\text{in } \Omega \backslash \Sigma, \\
%- \ljump \mu^\pm \partial_3 (u_*)_{1,2} \rjump - \ljump \mu^\pm \nabla_{x'} (u_*)_3 \rjump &= 0, && \text{on } \Sigma,   \label{982347608efgrgrt45tz527364} \\
%-2 \ljump \mu^\pm \partial_3 (u_*)_3 \rjump + \ljump \pi_* \rjump - \sigma \Delta_{x'} h_* \label{2837469327zfthzh8641} &= 0, && \text{on } \Sigma, \\
%\ljump u_* \rjump &= 0, && \text{on } \Sigma, \\
%\partial_t h_* \label{28374eerzrtztz69327864} -  e_3 \cdot u_*|_\Sigma +\ljump \partial_3 \mu_* \rjump &=  0, && \text{on } \Sigma, \\
%(-\nabla_{x'} h_*, 1)^\top \cdot \nu_{S_1} &= 0, && \text {on } \partial \Sigma, \\
%\Delta \mu_* &= 0, &&\text {in } \Omega \backslash \Sigma, \\
%\mu_*|_{\Sigma} - \sigma \Delta_{x'} h_* &= 0, && \text {on } \Sigma, \\
%\nu_{\partial\Omega} \cdot \nabla \mu_*|_{\partial\Omega} &= 0, && \text{on } \partial\Omega \backslash\Sigma, \\
%P_{S_1} \left( \mu^\pm (Du_* + Du_*^\top) \nu_{S_1} \right) &=0, && \text{on } S_1 \backslash \partial\Sigma, \\
%u_* \cdot \nu_{S_1} &= 0, && \text {on } S_1 \backslash \partial \Sigma, \\
%u_* &= 0, && \text{on } S_2, \\
%u_*(0) &= 0, && \text{on } \Omega \backslash \Sigma, \\
%h_* (0) &= h_0, && \text{on } \Sigma, \label{eq12affertertTRLb}
%\end{align}
%\end{subequations}
by functions $(u_*, \pi_*, h_*, \mu_*)$ in the proper regularity classes, cf. Section 3.2 in \cite{wilkehabil}. 
%By this reduction we can reduce the fixed point argument to functions having vanishing time trace at $t=0$. This is important since then the embedding constants are independent of $T$, whence we can show contraction estimates for small times.

Let us now introduce notation. Let
\begin{gather*}
_0 \mathbb E_u(T) := {_0 W}^1_r(0,T; L_r(\Omega)) \cap L_r(0,T;W^2_r(\Omega \backslash \Sigma)), \quad
\mathbb E_\pi(T) := L_r(0,T; \dot H^1_r(\Omega \backslash \Sigma)), \\
_0\mathbb E_q(T) :={_0 W}^{1/2 - 1/(2r)}_r(0,T;L_r(\Sigma)) \cap L_r(0,T;W^{1-1/r}_r(\Sigma)), \\
_0\mathbb E_h(T) := {_0 W}^1_p(0,T;W^{1-1/q}_q(\Sigma)) \cap L_p(0,T;W^{4-1/q}_q(\Sigma)),
 \end{gather*}
and $\mathbb E_\mu (T) := L_p(0,T;W^2_q(\Omega \backslash \Sigma))$. Furthermore, let
\begin{equation*}
_0 \mathbb E(T) := {_ 0\mathbb E_u(T)} \times  \mathbb E_\pi(T) \times {_0 \mathbb E_q(T)} \times {_0 \mathbb E_h(T)}  \times \mathbb E_\mu(T)  \cap \{ (u,\pi,q,h,\mu) :  q = \ljump \pi \rjump \}.
\end{equation*}
Moreover, let 
\begin{gather*}
\mathbb F_1(T) :=L_r(0,T;L_r(\Omega)), \quad \mathbb F_2(T) :=  L_r(0,T;W^1_r(\Omega \backslash \Sigma)), \\
\mathbb F_3(T) :=   {_0 W}^{1/2-1/(2r)}_r(0,T;L_r(\Sigma)) \cap L_r(0;T;W^{1-1/r}_r(\Sigma)) ,\\
\mathbb F_4(T) :=  {_0 W}^{1/2-1/(2r)}_r(0,T;L_r(\Sigma)) \cap L_r(0;T;W^{1-1/r}_r(\Sigma))  ,\\
\mathbb F_5(T) :=  {_0 W}^{1-1/(2r)}_r(0,T;L_r(\Sigma)) \cap L_r( 0,T; W^{2-1/r}_r(\Sigma))   ,\\
\mathbb F_6(T) :=  L_p(0,T;W^{1-1/q}_q(\Sigma))   ,\\
\mathbb F_7(T) :=  {_0 F}^{1-2/(3q)}_{pq}(0,T; L_q(\partial\Sigma)) \cap L_p(0,T; W^{3-2/q}_q(\partial\Sigma))   ,\\
\mathbb F_8(T) :=  L_p(0,T;L_q(\Omega))  ,\quad
\mathbb F_9(T) :=    L_p(0,T;W^{2-1/q}_q(\Sigma)) ,\\
\mathbb F_{10}(T) :=   L_p(0,T; W^{1-1/q}_q(\partial\Omega)) ,\\
\mathbb F_{11}(T) :=  {_0 W}^{1/2-1/(2r)}_r(0,T;L_r(S_1)) \cap L_r(0;T;W^{1-1/r}_r(S_1)  ,\\
\mathbb F_{12}(T) :=   {_0 W}^{1-1/(2r)}_r(0,T;L_r(S_1)) \cap L_r( 0,T; W^{2-1/r}_r(S_1)) ,\\
\mathbb F_{13}(T) :=  {_0 W}^{1-1/(2r)}_r(0,T;L_r(S_2)) \cap L_r( 0,T; W^{2-1/r}_r(S_2)) .
\end{gather*}
Let
\begin{equation} \label{dsddddcdcdcfdsfsdfsdfsdf}
_ 0 \mathbb F(T) := \times_{j=1}^{13} \mathbb F_{j}(T) \cap \{ (g_2,g_5,g_{12},g_{13}) \in W^1_r(\R_+; \hat H^{-1}_r(\Omega)) \}.
\end{equation}
Define a linear operator by the left hand side of \eqref{930487547650487gg5}, that is, define $\mathsf L : {_ 0 \mathbb E(T)} \pfeil {_0 \mathbb F(T)}$ via
\begin{equation*}
\mathsf L (u,\pi,q,h,\mu) := \begin{pmatrix}
\rho^\pm \partial_t u - \mu^\pm \Delta u + \nabla \pi \\
\div u \\
- \ljump \mu^\pm \partial_3 (u_1,u_2) \rjump - \ljump \mu^\pm \nabla_{x'} u_3 \rjump \\
-2 \ljump \mu^\pm \partial_3 u_3 \rjump + q - \sigma \Delta_{x'} h \\
\ljump u \rjump \\
\partial_t h - u_3|_\Sigma + \ljump \partial_3 \mu \rjump \\
(-\nabla_{x'} h, 1)^\top|_{\partial\Sigma} \cdot \nu_{S_1} \\
\Delta \mu \\
\mu|_\Sigma - \sigma \Delta_{x'} h \\
\nu_{\partial\Omega} \cdot \nabla \mu|_{\partial\Omega} \\
P_{S_1} \left( \mu^\pm (Du + Du^\top) \nu_{S_1} \right) \\
u|_{S_1} \cdot \nu_{S_1} \\
u|_{S_2}
\end{pmatrix}.
\end{equation*}
We collect the right hand side in the operator $\mathsf R : \mathbb E(T) \pfeil \mathbb F(T)$ defined by
\begin{equation*}
\mathsf R (u,\pi,q,h,\mu) := \begin{pmatrix}
a^\pm(h;D_x^2)(u,\pi) + \bar a(h,u) \\
G_d(u,h) \\
G_S(u,\pi,h)_{1,2} \\
G_S(u,\pi,h)_3 \\
0 \\
G_\Sigma(u,h,\mu) \\
0  \\
G_c(h,\mu) \\
G_\kappa(h) \\
G_N(h,\mu) \\
G^\pm_P(u,h) \\
0 \\
0
\end{pmatrix}.
\end{equation*}
Hereby $\mathbb E (T)$ and $\mathbb F(T)$ are defined similarly but without the trace properties at $t = 0$.
%For a definition of $a^\pm,\bar a, G_d,G_S,G_\Sigma, G_c,G_\kappa,G_N$ and $G_P^\pm$ we refer to \eqref{78236496723543643}-\eqref{78236496723543643e}.

It is now clear that for $h \in {_0 \mathbb E(T)}$ (which is a function having vanishing time trace) the compatibility condition $(-\nabla_{x'} h(t=0), 1)^\top|_{\partial\Sigma} \cdot \nu_{S_1} = 0$ is satisfied. Regarding the compatibility conditions for the Stokes system we refer to Section 3.1 in \cite{wilkehabil}. Therefore both operators are well defined.

Let $z := (u,\pi,q,h,\mu)$ and $z_* := (u_*,\pi_*,\ljump \pi_*\rjump,h_*,\mu_*)$ the reference solution as above. We can now rewrite the problem abstractly as
\begin{equation*}
\mathsf L(z+z_*) = \mathsf{R}(z+z_*), \quad z \in {_0 \mathbb E (T)}.
\end{equation*}
Note that we already know that $\mathsf L$ is invertible from $_0 \mathbb E(T)$ to $_0 \mathbb F(T)$ and the norms are independent of $T$. This renders the fixed point equation
\begin{equation*}
z = \mathsf L^{-1} (\mathsf{R}(z+z_*)  - \mathsf Lz_* ), \quad \text{ in } {_0 \mathbb E(T)}.
\end{equation*}
Define $\mathsf K : {_0 \mathbb E(T)} \pfeil  {_0 \mathbb E(T)}$ by means of $[ z \mapsto  \mathsf L^{-1} (\mathsf{R}(z+z_*)  - \mathsf Lz_* )]$.
% Note at this point that by construction $\mathsf{R}(z+z_*)  - \mathsf Lz_* $ and $ \mathsf{R}(z_1+z_*)-\mathsf{R}(z_2+z_*)$ both have vanishing traces at $t=0$ for all $z,z_1,z_2 \in {_0 \mathbb E(T)}.$
We now need to establish contraction estimates for $\mathsf R$. 
\begin{lemma} \label{23098476384756} We have
\begin{equation}
\begin{alignedat}{1}
|  \mathsf{R}(z_1+z_*)&-\mathsf{R}(z_2+z_*) |_{ _0 \mathbb F(T)} \\ &\leq C( T^\alpha+ |z_*|_{ \mathbb E(T)} + |z_1|_{_0 \mathbb E(T)} + |z_2|_{_0 \mathbb E(T)} ) | z_1 - z_2 |_{_0 \mathbb E(T)}, 
\end{alignedat}
\end{equation}
for some $\alpha > 0$ and for all $z_1,z_2 \in \mathsf B(r,0) \subset {_0 \mathbb E(T)}$, if $r > 0$ and $T = T(r)>0$ are sufficiently small.
\end{lemma}
Having these estimates at hand we proceed as in the proof of Theorem 5.1 in \cite{mulsekpaper12} to obtain a fixed point of $\mathsf K$ by Banach's contraction mapping principle by choosing $\varepsilon (T) >0$ in \eqref{98236736408763046ffdfg} small enough. This finishes the proof.
\end{proof}
\begin{proof}[Proof of Lemma \ref{23098476384756}.] 
 Let us first note that
\begin{gather}
[h \mapsto \Delta_h ] \in C^1(U; \mathcal B(W^2_r(\Omega \backslash \Sigma) ; L_r(\Omega) ) ), \label{8888882222}\\
[h \mapsto \nabla_h ] \in C^1(U; \mathcal B(W^k_r(\Omega \backslash \Sigma) ; W^{k-1}_r(\Omega \backslash \Sigma) ) ), \quad k = 1,2, \label{8888882222b}
\end{gather}
where $U \subset B^{4-1/q-3/p}_{qp}(\Sigma)$ is a sufficiently small neighbourhood of zero. This can be shown as in Lemma 3.4 in \cite{mulsekpaper12}.
 %Note in particular that since $h$ has vanishing trace at $t=0$, we have that $(\Delta_h - \Delta)|_{t=0} = 0$, the same for $\nabla_h - \nabla$.

We estimate every nonlinearity separately.
We recall that $a^\pm(h;D_x)(u,\pi) = \mu^\pm (\Delta_h - \Delta)u - (\nabla - \nabla_h)\pi$. Clearly,
\begin{equation*}
| ( \Delta_h - \Delta ) u |_{L_r(0,T;L_r(\Omega))} \leq |\Delta_h - \Delta|_{L_\infty(0,T; \mathcal B(W^2_r(\Omega\backslash \Sigma);L_r(\Omega)))} | u|_{L_r(0,T;W^2_r(\Omega \backslash \Sigma))}.
\end{equation*}
Using \eqref{8888882222} this gives
\begin{equation*}
| ( \Delta_h - \Delta ) u |_{L_r(0,T;L_r(\Omega))} \leq C|h|_{_0 \mathbb E(T)} |u|_{L_r(0,T;W^2_r(\Omega \backslash \Sigma))}.
\end{equation*}
The same arguments give
\begin{equation*}
| (\nabla_h - \nabla) \pi |_{L_r(0,T;L_r(\Omega))} \leq C|h|_{_0 \mathbb E(T)} | \pi|_{L_r(0,T;\dot H^1_r(\Omega\backslash \Sigma))},
\end{equation*}
since \eqref{8888882222b} is also true for the homogeneous counterparts $\dot H^k_r$ replacing $W^k_r$. Note that these estimates and the $C^1$-dependence of $h$ and the bilinear structure in $(u,\pi)$ of $a^\pm(h)(u,\pi)$ then automatically give rise to a contraction estimate of form
\begin{align} \label{049875634075}
| &a^\pm(h_1)(u_1,\pi_1) - a^\pm(h_2)(u_2,\pi_2) |_{L_r(0,T;L_r(\Omega))} \leq \\ &\leq C |h_1-h_2|_{_0 \mathbb E(T)} \left( |u_1 - u_2|_{L_r(0,T;W^2_r(\Omega \backslash \Sigma))} + | \pi_1-\pi_2|_{L_r(0,T;\dot H^1_r(\Omega\backslash \Sigma))} \right), \nonumber
\end{align}
valid for all $h_1,h_2 \in U, u_1,u_2 \in L_r(0,T;W^2_r(\Omega \backslash \Sigma)), \pi_1, \pi_2 \in L_r(0,T;\dot H^1_r(\Omega\backslash \Sigma))$. Indeed, we have that
$
a^\pm \in C^2 \left( \mathbb E_h(T)  \times \mathbb E_u(T) \times \mathbb E_\pi (T) ; L_r(0,T;L_r(\Omega))  \right)
$
and 
$
a^\pm(0) =0,$ $Da^\pm (0) = 0$.
Alternatively, one can explicitly estimate the difference and end up with \eqref{049875634075}. Before we estimate $\bar a(u,h,\eta) := Du \cdot \partial_t \Theta^{-1}_h + (u \cdot \nabla_h)u + (\rho^+ - \rho^-)(\nabla_h \eta \cdot \nabla_h)u$, some remarks are in order.
Firstly, $Du \cdot \partial_t \Theta_h^{-1} = -\chi \partial_t h(1+h\chi')^{-1} \partial_3 u$, see \cite{pruessbuch}.

Contracting the (transformed) convection term $(u \cdot \nabla_h ) u$ is easy due to the fact that $\mathbb E_u(T) \into L_\infty(0,T;L_\infty(\Omega)) \cap L_\infty(0,T;W^1_r(\Omega\backslash \Sigma))$. More precisely,
\begin{align*}
| (u \cdot \nabla_h) &u |_{L_r(0,T;L_r(\Omega))} \leq \\ & T^{1/r} |u|_{L_\infty(0,T;L_\infty(\Omega))} | \nabla_h|_{L_\infty(0,T;\mathcal B(W^1_r(\Omega \backslash \Sigma) ; L_r(\Omega)))} | u |_{L_\infty(0,T;W^1_r(\Omega \backslash  \Sigma))}.
\end{align*}
Regrading the other terms we recall that $Du \in L_{2r}(0,T;L_\infty(\Omega))$ due to $r>3$. We then get by H\"older inequality that
\begin{equation*} \begin{alignedat}{1}
| Du &\cdot \partial_t \Theta_h^{-1} |_{L_r(0,T;L_r(\Omega))} \leq \\ & | \chi (1+ h \chi')^{-1} |_{L_\infty(0,T;L_\infty(\Omega))} | Du |_{L_{p_1}(0,T;L_\infty(\Omega))} | \partial_t h |_{L_p(0,T;L_{r}(\Sigma))} |1|_{L_{p_0}(0,T;L_\infty(\Omega))}, \end{alignedat}
\end{equation*}
where $1 < p_0, p_1 < \infty$ are such that
\begin{equation} \label{34876504365943}
\frac{1}{r} = \frac{1}{p_1} +\frac{1}{p} +\frac{1}{p_0}.
\end{equation}
By choice of $q < 2$ and $3 < r < 7/2$, Sobolev's embedding theorem gives $$W^{1-1/q}_q(\Sigma) \into L_r(\Sigma).$$ Choosing $p_1 = 2r$ and recalling $p > 2r$ gives that there is some $1 < p_0 < \infty$ such that \eqref{34876504365943} is fulfilled. Note that these estimates are not optimal but sufficient in our case. We then obtain that there is some $\varepsilon = \varepsilon(p,q,r) > 0$ such that
\begin{equation*}
| Du \cdot \partial_t \Theta_h^{-1} |_{L_r(0,T;L_r(\Omega))} \leq CT^\varepsilon | u |_{\mathbb E_u(T)} | h |_{\mathbb E_h(T)}.
\end{equation*}
Furthermore,
\begin{equation*}
| (\nabla_h \eta \cdot \nabla_h) u |_{L_r(0,T;L_r(\Omega))} \leq | \nabla_h \eta |_{L_p(0,T;L_r(\Omega))} | \nabla_h u |_{L_{p_1}(0,T;L_\infty(\Omega))} |1|_{L_{p_0}(0,T;L_\infty(\Omega))},
\end{equation*}
where $p_0, p_1$ are as above. Again by Sobolev embedding, $W^1_q(\Omega \backslash \Sigma) \into L_r(\Omega)$, whence
\begin{equation*}
| (\nabla_h \eta \cdot \nabla_h) u |_{L_r(0,T;L_r(\Omega))} \leq | \nabla_h \eta |_{L_p(0,T;W^1_q(\Omega\backslash\Sigma))} | \nabla_h u |_{L_{p_1}(0,T;L_\infty(\Omega))} |1|_{L_{p_0}(0,T;L_\infty(\Omega))}.
\end{equation*}
In view of \eqref{8888882222b}, these estimates together with the smooth dependence $$\bar a \in C^\infty( \mathbb E_u(T) \times \mathbb E_h(T) ; \mathbb F_1(T) )$$ as well as $\bar a(0,0,0) = 0$ and $D\bar a(0,0,0) = 0$ give rise to contraction estimates for $\bar a$.

For $G_d (u,h) := (\div - \div_h)u$, the estimate in $\mathbb F_2(T)$ is straightforward, 
\begin{equation*}
| G_d(u,h)|_{L_r(0,T;W^1_r(\Omega \backslash \Sigma))} \leq | \nabla - \nabla_h|_{L_\infty(0,T;\mathcal B(W^1_r(\Omega \backslash \Sigma) ; L_r(\Omega)))} | u|_{L_r(0,T;W^2_r(\Omega \backslash \Sigma))},
\end{equation*}
where we used that $G_d(u,h) = \operatorname{Tr} (\nabla - \nabla_h) u$.

The contraction estimates for $G_c(h,\eta) := (\Delta - \Delta_h)\eta = (\div \nabla - \div_h \nabla_h)\eta$ and $G_N(h,\eta) := \nu_{\partial\Omega} \cdot (\nabla-\nabla_h)\eta$ easily stem from \eqref{8888882222}-\eqref{8888882222b}  with $q$ replacing $r$, see also \cite{mulsekpaper12}. Note that there the contraction estimates for $G_\kappa (h) := \sigma( K(h) - \Delta_{x'} h)$ are already proven in a far more general setting. In this graph situation case we can give a much easier proof. Recall that in this case $K(h) = \div_{x'}( \nabla_{x'}  h (1+|\nabla_{x'} h|^2 )^{-1/2} )$, whence
\begin{equation} \label{2934857304985}
G_\kappa(h) = \left( 1- \frac{1}{\sqrt{  1+|\nabla_{x'} h|^2} } \right)\Delta_{x'} h + \nabla_{x'} h \cdot \nabla_{x'}\left( \frac{1}{ \sqrt{ 1 + |\nabla_{x'} h|^2 }  } \right).
\end{equation}
Again using the product estimate
\begin{equation*}
| \nabla h \cdot \nabla^2 h|_{ \mathbb F_9(T)} \leq C | \nabla h |_{L_\infty(0,T;B^{3-1/q-3/p}_{qp}(\Sigma))} |\nabla^2 h|_{\mathbb F_9(T)} \leq C |h|^2_{ _0 \mathbb E_h(T)}
\end{equation*}
and the fact that $G_\kappa \in C^\infty( {_0 \mathbb E_h(T)} ; \mathbb F_9(T) )$, $G_\kappa(0) = 0, DG_\kappa(0) = 0$
ensure the contraction property of $G_\kappa$.

Regarding $G^\pm_P(u,h)$ it is shown in Section 3.1 in \cite{wilkehabil}, that
\begin{equation} \label{3940857034765380475643}
\begin{alignedat}{1}
G^\pm_P&(u,h) \\ &= P_{S_1} \left[  \frac{1}{1+\chi' h} \left( \chi \partial_3 u ((-\nabla_{x'} h,1)^\top \cdot \nu_{S_1}) + \begin{pmatrix}
\chi \nabla_{x'} h \\ \chi' h
\end{pmatrix}  \partial_3 ( u \cdot \nu_{S_1} ) \right)   \right].
\end{alignedat}
\end{equation}
Therefore, due to the fact that $u \cdot \nu_{S_1} = 0 $ on $S_1 \backslash \partial\Sigma$ and $(-\nabla_{x'} h,1)^\top \cdot \nu_{S_1} = 0$ on $\partial\Sigma \times (L_1,L_2)$, the nonlinearity $G^\pm_P(u,h)$ vanishes for the solution. Hence we may replace $G_P^\pm (u,h)$ by zero in the definition of $\mathsf R$.

Now, for $G_S(h,u,\pi) $ we split $ G_S(h,u,\pi) = G_S^S(h,u,\pi) + G_S^\kappa(h)$, where
\begin{align*}
G_S^S(h,u,\pi) &:= 
\ljump \mu^\pm \left( (D\Theta_h - I)Du + Du^\top (D\Theta_h - I)^\top ) \right) \rjump \nu_{\Sigma_h} +    \\
&+ \ljump \left( \mu^\pm (Du+Du^\top) - \pi I \right) (e_3 - \nu_{\Sigma_h} ) \rjump, \\
G_S^\kappa(h) &:= \sigma( K(h) \nu_{\Sigma_h} - \Delta_{x'} h e_3 ).
\end{align*}
Regarding the estimates of $G_S^S(h,u,\pi)$ we refer to \cite{prsimNSST}. Note that due to Remark 1.2. (c) in \cite{prsimNSST} we may use these results since $r>3$ and $ \mathbb E_h(T) \into BUC([0,T];C^2(\Sigma))$.

Considering $G_S^\kappa(h)$ we may write $G_S^\kappa(h) = G_\kappa(h)e_3 + K(h)(\nu_{\Sigma_h} - e_3)$ and estimate each term separately. In particular, we have to control terms of the form $\nabla h \cdot \nabla^2 h$ in the norm of $\mathbb F_3(T) = {_0 W}^{1/2-1/(2r)}_r(0,T;L_r(\Sigma)) \cap L_r(0,T; W^{1-1/r}_r(\Sigma))$. This stems from the observation in \eqref{2934857304985}. Now, by Theorem \ref{thmr} we already know that the space in which second derivatives of $h$ live in embeds into $\mathbb F_3(T)$. 
We may now use the product estimate of Proposition 5.7 in \cite{meyriesveraarpointwise} to obtain
\begin{align*}
| \nabla h \cdot \nabla^2 h |_{W^{1/2-1/(2r)}_r(0,T;L_r(\Sigma))} &\lesssim | \nabla h|_{L_\infty(0,T;L_\infty(\Sigma))} | \nabla^2 h |_{W^{1/2-1/(2r)}_r(0,T;L_r(\Sigma))} + \\
&+ | \nabla h |_{W^{1/2-1/(2r)}_r(0,T;L_\infty(\Sigma))} | \nabla^2 h |_{L_\infty(0,T;L_r(\Sigma))}.
\end{align*}
Furthermore,
\begin{equation*}
| \nabla h \cdot \nabla^2 h |_{L_r(0,T;W^{1-1/r}_r(\Sigma))} \lesssim | \nabla h |_{L_\infty(0,T;C^1(\Sigma))} | \nabla^2 h |_{L_r(0,T;W^{1-1/r}_r(\Sigma))}.
\end{equation*}
These estimates show that the product terms of form $\nabla h \cdot \nabla^2 h$ are well defined in $\mathbb F_3(T)$. 

These observations allow us to conclude contraction estimates for $G_S^\kappa$ since again $G_S^\kappa (0) = 0, DG_S^\kappa (0) = 0$.

Regarding $G_\Sigma(u,h,\mu) = u|_\Sigma \cdot (-\nabla_{x'} h,0)^\top - \ljump e_3\cdot (\nabla - \nabla_h) \mu \rjump - \ljump (-\nabla_{x'} h,0)^\top \cdot \nabla_h \mu \rjump,$ the last two terms can be controlled as before. Clearly the first term is smooth in $(u,h)$ and quadratic and the bound
\begin{equation*}
|  u|_\Sigma \cdot (-\nabla_{x'} h,0)^\top |_{L_p(0,T;W^{1-1/q}_q(\Sigma))} \leq T^{1/p} |u|_{L_\infty(0,T;W^1_q(\Omega\backslash\Sigma))} | \nabla h|_{L_\infty(0,T;C^1(\Sigma))}
\end{equation*}
renders contraction estimates also for $G_\Sigma$. This concludes the proof of the contraction estimates.
\end{proof}

\section{Qualitative behaviour}

In this section we investigate the long-time behaviour of solutions starting close to equilibria. By a study of the spectrum of the linearization we will show that solutions starting close to certain equilibria converge to an equilibrium solution at an exponential rate.

% Since, among other things, the Hanzawa transformation onto a fixed domain of the free boundary problem destroys the divergence-free condition of the velocity field we will have to split the solution into two parts: the one which is divergence free and the part which is not. We follow the lines of [AbelsWilke] and [WilkeHabil].

Let us again consider the case of a cylindrical container $\Omega = \Sigma \times (L_1,L_2)$, where $ -\infty < L_1 < 0 < L_2 <\infty$ and $\Sigma \subset \R^2$ is open, bounded and has smooth boundary.
We want to study stability properties of 
\begin{equation} \label{9348576dfgdfgkkkkkrtzkrtz9765}
\begin{alignedat}{2}
\rho \partial_t  u  - \mu \Delta u + \div[ (\rho u + \ljump \rho \rjump \nabla \eta ) \otimes u ] + \nabla p  &= 0, &&\text{in } \Omega \backslash \Gamma(t), \\
\operatorname{div} u &= 0, &&\text{in } \Omega \backslash \Gamma(t), \\
- \ljump \mu (Du + Du^\top) \rjump \nu_{\Gamma (t)} + \ljump p \rjump \nu_{\Gamma (t)} &= \sigma H_{\Gamma(t)} \nu_{\Gamma(t)}, && \text{on } \Gamma(t), \\
\ljump u \rjump &= 0, && \text{on } \Gamma(t), \\
 V_{\Gamma(t)} -u|_{\Gamma(t)} \cdot \nu_{\Gamma(t)} &= - \ljump \nu_{\Gamma(t)} \cdot \nabla \eta \rjump, \quad && \text{on } \Gamma(t), \\
\nu_{\Gamma(t)} \cdot \nu_{S_1} &= 0, && \text {on } \partial \Gamma (t), \\
\Delta \eta &= 0, &&\text {in } \Omega \backslash \Gamma(t), \\
\eta|_{\Gamma(t)} &= \sigma H_{\Gamma(t)} , && \text {on } \Gamma(t), \\
\nu_{\partial\Omega} \cdot \nabla \eta|_{\partial\Omega} &= 0 , && \text{on } \partial\Omega \backslash\Gamma(t), \\
P_{S_1} \left( \mu (Du + Du^\top) \nu_{S_1} \right) &= 0, && \text{on } S_1 \backslash \partial\Gamma(t), \\
u \cdot \nu_{S_1} &= 0, && \text {on } S_1 \backslash \partial \Gamma(t), \\
u &= 0, && \text{on } S_2, \\
u(0) &= u_0, && \text{on } \Omega \backslash \Gamma(0), \\
\Gamma (0) &= \Gamma_0.
\end{alignedat}
\end{equation}
We recall that $\rho := \rho^+ \chi_{\Omega^+ (t)} + \rho^- \chi_{\Omega^-(t)}$ and $\mu := \mu^+ \chi_{\Omega^+ (t)} + \mu^- \chi_{\Omega^-(t)}$.

\subsection{Equilibria and spectrum of the linearization}
We note that the pressure $p$ as well as the chemical potential $\mu$ may be reconstructed by the semiflow $(u(t),\Gamma(t))$ as follows. For given $\Gamma(t)$ we can solve the two-phase elliptic problem
\begin{align*}
\Delta \eta &= 0, && \text{in } \Omega \backslash \Gamma(t), \\
\eta|_{\Gamma(t)} &= \sigma H_{\Gamma(t)}, && \text{on } \Gamma(t), \\
n_{\partial\Omega} \cdot \nabla \eta|_{\partial\Omega} &= 0, &&\text{on } \partial\Omega,
\end{align*}
and the weak transmission problem
\begin{align*}
(\nabla p / \rho | \nabla \phi)_{L_2(\Omega)} &= ([\mu/\rho] \Delta u - u \cdot \nabla u | \nabla \phi )_{L_2(\Omega)}, \quad && \text{for all } \phi \in W^1_{r'}(\Omega), \\
\ljump p \rjump &=  \ljump \mu (Du+Du^\top) \nu_{\Gamma(t)} \cdot \nu_{\Gamma(t)} \rjump + \sigma H_{\Gamma(t)}, && \text{on } \Gamma(t),
\end{align*}
where $r' = r/(r-1)$, cf. Lemma A.7 in \cite{wilkehabil}.
%Here we want to point out that $(\partial_t u | \nabla \phi)_2 = d/dt (u | \nabla \phi)_2 = d/dt ( \div u | \phi )_2 = 0$, due to the boundary conditions.
Therefore we may concentrate on the set of equilibria $\mathcal E$ for the semiflow $(u(t), \Gamma(t))$.
Note that
the set of equilibria for \eqref{9348576dfgdfgkkkkkrtzkrtz9765} is given by $\mathcal E = \{ (u,\Gamma) : u = 0, H_\Gamma = \text{const.}\}$. In particular, also $\mu$ is constant, $p$ is constant in the two phases of $\Omega \backslash \Gamma$ and also the jump $\ljump p \rjump$ is constant on $\Gamma$.
%\begin{proof}
%Say we have a stationary solution $(u,\Gamma)$.
%Testing the equations and invoking the transmission and boundary conditions gives
%\begin{equation}
%| (\mu^\pm)^{1/2} (Du + Du^\top ) |_{L_2(\Omega)}^2 +  | \nabla \mu |^2_{L_2(\Omega)} = 0.
%\end{equation} 
%Hence $u = 0$ by Korn's inequality and $\mu$ is constant. Consequently, $p$ is constant with possibly different values in the phases and $H_\Gamma$ is constant. 
%We want to quickly explain why the boundary terms of $Du$ vanish in \eqref{92376982476}. The boundary integrals are of form
%\begin{equation} \label{3948056304754386508473}
%\int_{\partial\Omega} \mu^\pm (Du + Du^\top ) \nu_{\partial\Omega} \cdot u|_{\partial\Omega} d\partial\Omega.
%\end{equation}
%By equation \eqref{903478650etrhj83476}, the integral vanishes on $S_2$. Regarding $S_1$ we note that by applying the projection $P_{S_1} = I - \nu_{S_1} \otimes \nu_{S_1}$,
%\begin{equation*}
%\mu^\pm (Du + Du^\top) \nu_{S_1} \cdot u |_{S_1} = P_{S_1}(\mu^\pm (Du+Du^\top) \nu_{S_1} ) u|_{S_1} + [ \nu_{S_1} \otimes \nu_{S_1} ] ( \mu^\pm (Du+ Du^\top) \nu_{S_1} ) u|_{S_1}.
%\end{equation*}
%The first term vanishes on $S_1$ due to the equations, for the second one we directly calculate
%\begin{equation*}
% [ \nu_{S_1} \otimes \nu_{S_1} ] ( \mu^\pm (Du+ Du^\top) \nu_{S_1} ) u|_{S_1} = (\nu_{S_1} \cdot u|_{S_1} ) [ (\nu_{S_1} \otimes \nu_{S_1}) : \mu^\pm (Du+Du^\top) ], 
%\end{equation*}
%whence this also vanishes on $S_1$ and renders the integral \eqref{3948056304754386508473} to be zero. This concludes the proof.
%\end{proof}
\begin{remark}
We want to point out that in the special case when $\Gamma$ is a $C^2$-graph of a function $h$ over $\Sigma$, we can even deduce that $H_\Gamma = 0$ and $h$ is constant. A proof of this can be found in \cite{mulsekpaper12}.

%Let us give a proof of this. We already know that $H_\Gamma$ is constant. By a shifting argument, we may assume that $h$ is mean value free on $\Sigma$. Testing $H_\Gamma$ with $h$ in $L_2(\Sigma)$ gives
%\begin{equation}
%0 = (H_\Gamma, h)_{L_2(\Sigma)} = - \int_\Sigma \frac{ | \nabla h |^2} { \sqrt{1 + | \nabla h |^2} } dx.
%\end{equation}
%Hence $h$ is constant and $H_\Gamma = 0$.
\end{remark}
We again now work in the graph situation, that is, we assume the free interface $\Gamma(t)$ is a graph of a height function $h$ over $\Sigma$.

The linearization of the transformed Two-phase Navier-Stokes/Mullins-Sekerka problem \eqref{9348576dfgdfgkkkkkrtzkrtz9765} around the trivial equilibrium $(0, \Sigma) \in \mathcal E$ induces us to study the problem
\begin{equation}\label{03947fghsfghsrh534}
\begin{alignedat}{2} 
\rho \partial_t u - \mu \Delta u  + \nabla p &= f_u, &&\text{in } \Omega \backslash \Sigma, \\
\operatorname{div} u &= 0, &&\text{in } \Omega \backslash \Sigma, \\
 - \ljump \mu (Du + Du^\top) \rjump e_3 + \ljump p \rjump e_3 + \sigma \Delta_{x'} h e_3 &=  0, && \text{on } \Sigma,  \\
\ljump u \rjump &= 0, && \text{on } \Sigma, \\
  \partial_t h - u_3 + \ljump \partial_3 \eta \rjump &=  f_h , && \text{on } \Sigma, \\
(\nabla_{x'} h, -1)^\top \cdot \nu_{S_1} &= 0 && \text {on } \partial \Sigma, \\
\Delta \eta &= 0, &&\text {in } \Omega \backslash \Sigma, \\
\eta|_{\Sigma} + \sigma \Delta_{x'} h &= 0, && \text {on } \Sigma, \\
\nu_{\partial\Omega} \cdot \nabla \mu|_{\partial\Omega} &= 0, && \text{on } \partial\Omega \backslash\Sigma, \\
P_{S_1} \left( \mu (Du + Du^\top) \nu_{S_1} \right) &= 0, && \text{on } S_1 \backslash \partial\Sigma, \\
u \cdot \nu_{S_1} &= 0, \qquad\qquad\qquad&& \text {on } S_1 \backslash \partial \Sigma, \\
u &= 0, && \text{on } S_2, \\
u(0) &= u_0, && \text{in } \Omega \backslash \Sigma, \\
h(0) &= h_0, &&\text{on } \Sigma, 
\end{alignedat}
\end{equation}
where $f_h$ is assumed to be mean value free.
Let us note the following observations. Integrating equation $\eqref{03947fghsfghsrh534}_5$ over $\Sigma$ yields 
$
\int_\Sigma h(t) dx = \int_\Sigma h_0 dx$ for all $t \in \R_+$.
In other words, whenever $h_0$ and $f_h$ are mean value free, the solution $h$ will stay mean value free for all times.
Furthermore, applying $P_\Sigma = I - e_3 \otimes e_3$ to equation $\eqref{03947fghsfghsrh534}_3$ directly yields that $P_\Sigma ( \ljump \mu^\pm (Du + Du^\top) \rjump e_3 ) = 0$ on $\Sigma$.

 We want to write system \eqref{03947fghsfghsrh534} as an abstract evolution equation. To this end let
\begin{equation*}
X_0 := L_{r,\sigma} (\Omega) \times W^{1-1/q}_q(\Sigma), \quad X_1 := (L_{r,\sigma}(\Omega) \cap W^2_r(\Omega \backslash \Sigma) ) \times W^{4-1/q}_q(\Sigma),
\end{equation*}
and define a linear operator $A : D(A) \subset X_1 \pfeil X_0$ by
\begin{equation*}
A(u,h) := (- [\mu/\rho] \Delta u + \nabla p , - u_3 + \ljump \partial_3 \eta \rjump )
\end{equation*}
with domain
\begin{align*}
D(A) := \{ (u,h) &\in X_1 :  \ljump u \rjump = 0 \text{ on } \Sigma, \;
P_{S_1} \left( \mu^\pm (Du + Du^\top) \nu_{S_1} \right) = 0 \text{ on } S_1\backslash \partial\Sigma, \\
&u \cdot \nu_{S_1} = 0 \text{ on } S_1 \backslash \partial\Sigma, \; u = 0 \text{ on } S_2, \\
&P_\Sigma ( \ljump \mu^\pm (Du + Du^\top) \rjump e_3 ) = 0 \text{ on } \Sigma, \;
(\nabla_{x'} h, -1)^\top \cdot \nu_{S_1} = 0 \text{ on } S_1 \}.
\end{align*}
Here, $p \in \dot H^1_r (\Omega \backslash \Sigma)$ solves the weak transmission problem
\begin{subequations}
\begin{align*}
(\nabla p /\rho | \nabla \phi)_{L_2(\Omega)} &= ([\mu/\rho] \Delta u  | \nabla \phi )_{L_2(\Omega)}, \quad && \text{for all } \phi \in W^1_{r'}(\Omega), \\
\ljump p \rjump &= \sigma \Delta_{x'} h + ( \ljump \mu^\pm (Du+Du^\top)  \rjump e_3 | e_3 )_{L_2(\Sigma )}, && \text{on } \Sigma,
\end{align*}
\end{subequations}
cf. Lemma A.7 in \cite{wilkehabil} and $\eta \in W^2_q(\Omega \backslash \Sigma)$ solves the elliptic problem
\begin{align*}
\Delta \eta &=0, && \text{in } \Omega \backslash \Sigma, \\
\eta|_\Sigma + \sigma \Delta_{x'} h &= 0, && \text{on } \Sigma, \\
\partial_\nu \eta &= 0, &&\text{on } \partial\Omega.
\end{align*} 
As in \cite{abelswilke} and \cite{wilkehabil}, we will sometimes make use of the notation via solution operators, that is,
\begin{equation*}
\nabla p / \rho= T_1[(\mu/\rho) \Delta u ] + T_2[ \sigma \Delta_{x'} h + ( \ljump \mu (Du+Du^\top)  \rjump e_3 | e_3 )_{L_2(\Sigma )}   ],
\end{equation*}
%where $T_1 : L_r(\Omega) \pfeil L_r(\Omega)$ and $T_2 : W^{1-1/r}_r(\Sigma) \pfeil L_r(\Omega)$ are bounded, linear operators, 
cf. Lemma A.7 in \cite{wilkehabil}. Note that $\Delta_{x'} h \in W^{2-1/q}_q(\Sigma) \into W^{1-1/r}_r(\Sigma)$ for $h \in W^{4-1/q}_q(\Sigma)$ by Sobolev embedding, since $r < 3q/(3-q)$.

We can then rewrite problem \eqref{03947fghsfghsrh534} in a more compact form as
\begin{equation} \label{8937046580354765}
\dot z(t) + Az(t) = f(t), \; \; t \in \R_+, \quad \; z(0) = z_0,
\end{equation}
where $z:= (u,h)$, $f := (f_u,f_h)$ and $z_0 := (u_0,h_0)$.
We can now show a similar result as in \cite{abelswilke} about properties of the operator $A$.
\begin{lemma} \label{39840567349576340576834065734}
Let $n = 2,3$, $(p,q,r)$ as in Theorem \ref{thmr}, $\rho^\pm$, $\mu^\pm, \sigma > 0$ constant and $X_0$ and $A$ as above. Then the following statements are true.
\begin{enumerate}  %\itemsep-0.3em
%\item For every $f \in L_r(J;L_r(\Omega)) \times L_p(J;W^{1-1/q}_q(\Sigma))$ and $z_0 \in W^{2-2/r}_{r,\sigma}(\Omega \backslash \Sigma) \times B^{4-1/q-3/p}_{qp}(\Sigma)$ there is a unique solution $z=(u,h)$ of \eqref{8937046580354765} with $u \in W^1_r(J;L_r(\Omega)) \cap L_r(J;W^2_r(\Omega \backslash \Sigma))$ and $h \in W^1_p(J;W^{1-1/q}_q(\Sigma)) \cap L_p(J;W^{4-1/q}_q(\Sigma))$.
\item The linear operator $-A$ generates an analytic $C_0$-semigroup $e^{-At}$ in $X_0$.
\item The spectrum $\sigma(-A)$ consists of countably many eigenvalues with finite algebraic multiplicity.
\item $\lambda = 0$ is a semi-simple eigenvalue with multiplicity $1$ and $X_0 = N(A) \oplus R(A)$.
\item $\sigma (-A)\backslash  \{ 0 \} \subset \mathbb C_- := \{ z \in \mathbb C : \operatorname{Re} z < 0 \}$.
\item The kernel $N(A)$ is isomorphic to the tangent space $T_{(0,\Sigma)} \mathcal E$ of $\mathcal E$ at the trivial equilibrium $(0,\Sigma) \in \mathcal E$ and is given by $N(A) = \{ (u,h) : u = 0, h = const. \}$.
\item The restriction of $e^{-At}$ to $R(A)$ is exponentially stable.
\end{enumerate}
\end{lemma}
\begin{proof}
The first assertion follows from Theorem \ref{maxregodfpgudfphgu} and the proof of %Proposition 5.1 in \cite{abelswilke} and 
Proposition 1.2 in \cite{prmaxreglpspace}.
Since $D(A)$ compactly embeds into $X_0$, the resolvent of $A$ is compact and therefore the spectrum of $A$ consists only of countably many eigenvalues with finite multiplicity. By classical results, it does not depend on $q$ and $r$, cf. \cite{arendt}, \cite{engelnagel}. So let $\lambda \in \sigma(-A)$ be an eigenvalue with eigenfunctions $(u,h) \in D(A)$. The corresponding eigenvalue problem reads as 
\begin{equation} \label{03947fghsfghsrrtrh534}
\begin{alignedat}{2}
 \lambda \rho u - \mu \Delta u  + \nabla p &= 0, &\text{in } \Omega \backslash \Sigma,  \\
\Delta \eta = 0, \quad \operatorname{div} u &= 0, &\text{in } \Omega \backslash \Sigma,  \\
 - \ljump \mu (Du + Du^\top) \rjump e_3 + \ljump p \rjump e_3 - \sigma \Delta_{x'} h e_3 &=  0, & \text{on } \Sigma,  \\
  \lambda h - u_3 + \ljump \partial_3 \eta \rjump &=  0 , & \text{on } \Sigma, \\
(\nabla_{x'} h, -0)^\top \cdot \nu_{S_1} &= 0 & \text {on } \partial \Sigma,  \\
\ljump u \rjump = 0, \quad \ljump \eta \rjump = 0, \quad \eta|_{\Sigma} + \sigma \Delta_{x'} h &= 0, & \text {on } \Sigma, \\
\nu_{\partial\Omega} \cdot \nabla \mu|_{\partial\Omega} &= 0, & \text{on } \partial\Omega \backslash\Sigma, \\
u \cdot \nu_{S_1} = 0, \quad P_{S_1} \left( \mu (Du + Du^\top) \nu_{S_1} \right) &= 0, \quad\quad & \text{on } S_1 \backslash \partial\Sigma, \\
u &= 0, & \text{on } S_2. 
\end{alignedat}
\end{equation}
Testing equation $\eqref{03947fghsfghsrrtrh534}_1$ with $u$ in $L_2(\Omega)$ and invoking boundary and transmission conditions yields
\begin{equation} \label{9083746589374650387456}
\lambda |\rho^{1/2} u|_{L_2(\Omega)} + | \mu^{1/2} (Du+Du^\top) |_{L_2(\Omega)} + \sigma \bar\lambda | \nabla_{x'} h|_{L_2(\Sigma)}^2 + | \nabla \eta |_{L_2(\Omega)} = 0.
\end{equation}
Let $\lambda = 0$. Then $u = 0$ by Korn's inequality and $\eta = const.$, whence $\Delta_{x'} h$ is constant on $\Sigma$. An integration over $\Sigma$ together with the boundary condition $\eqref{03947fghsfghsrrtrh534}_6$ yields that
%\begin{equation}
%\int_\Sigma \Delta_{x'} h(x') dx' = \int_{\partial\Sigma} \nabla_{x'} h(x') \cdot n_{\partial\Sigma}(x') d\partial \Sigma = 0,
%\end{equation}
%whence 
$\Delta_{x'} h = 0$ on $\Sigma$. Hence $h$ 
%solves the homogeneous Neumann problem
%\begin{equation}
%\Delta_{x'} h = 0, \; \text{ on } \Sigma, \qquad \nabla_{x'} h \cdot n_{\partial\Sigma} = 0, \; \text{ on } \partial\Sigma,
%\end{equation}
%whence $h$
has to be constant. We obtain that the kernel $N(A)$ is one-dimensional and $N(A) = \{ (u,h) : u = 0, h = const. \}.$ 
%Since the real parts of $\lambda$ and $\bar \lambda$ coincide, we also obtain that $\sigma(-A) \cap i\R = \{ 0 \}$. In particular, taking real parts in \eqref{9083746589374650387456} yields
%\begin{equation} \label{9083746fgsh589374650387456}
%\operatorname{Re} \lambda |u|_{L_2(\Omega)} + | (\mu^\pm)^{1/2} (Du+Du^\top) |_{L_2(\Omega)} + \sigma \operatorname{Re} \lambda | \nabla_{x'} h|_{L_2(\Sigma)}^2 + | \nabla \eta |_{L_2(\Omega)} = 0,
%\end{equation}
%whence $\operatorname{Re} \lambda < 0$.
Taking real parts in \eqref{9083746589374650387456} yields $\operatorname{Re} \lambda \leq 0$. We also easily obtain that $\sigma(-A) \cap i\R = \{ 0 \}$, hence
$\sigma(-A) \backslash \{ 0 \}  \subset \mathbb C_-$. Next we show that the eigenvalue $\lambda = 0$ is semi-simple. Pick $z = (u,h) \in N(A^2)$. Then $z_1 := Az \in N(A)$, hence $z_1 = (0,h_1)$ and $h_1$ is constant. The problem for $z = (u,h)$ now reads as
\begin{equation}\label{03947fghsfdfdfdfghs5555555555555rrtrh534}
\begin{alignedat} {2}
- \mu \Delta u  + \nabla p &= 0, &\text{in } \Omega \backslash \Sigma,  \\
\operatorname{div} u &= 0, &\text{in } \Omega \backslash \Sigma,  \\
 - \ljump \mu (Du + Du^\top) \rjump e_3 + \ljump p \rjump e_3 - \sigma \Delta_{x'} h e_3 &=  0, & \text{on } \Sigma,  \\
\ljump u \rjump &= 0, & \text{on } \Sigma, \\
  - u_3 + \ljump \partial_3 \eta \rjump &=  h_1 , & \text{on } \Sigma, \\
(\nabla_{x'} h, -1)^\top \cdot \nu_{S_1} &= 0 & \text {on } \partial \Sigma,  \\
\Delta \eta &= 0, &\text {in } \Omega \backslash \Sigma, \\
\eta|_{\Sigma} + \sigma \Delta_{x'} h &= 0, & \text {on } \Sigma, \\
\nu_{\partial\Omega} \cdot \nabla \mu|_{\partial\Omega} &= 0, \quad\qquad\qquad& \text{on } \partial\Omega \backslash\Sigma, \\
P_{S_1} \left( \mu (Du + Du^\top) \nu_{S_1} \right) &= 0, & \text{on } S_1 \backslash \partial\Sigma, \\
u \cdot \nu_{S_1} &= 0, & \text {on } S_1 \backslash \partial \Sigma.\\
u &= 0, & \text{on } S_2. 
\end{alignedat}
\end{equation}
Integrating $\eqref{03947fghsfdfdfdfghs5555555555555rrtrh534}_5$ over $\Sigma$ and using the fact that $h_1$ is constant yields that $h_1 = 0$, since the other terms are mean value free. This yields that $(u,h) \in N(A)$, whence $N(A^2) \subset N(A)$. Since $A$ has compact resolvent, $R(A)$ is closed in $X_0$ and $\lambda = 0$ is a pole of $(\lambda - A)^{-1}$. Therefore $\lambda = 0$ is semi-simple, cf. \cite{lunardioptimal}, and $X_0 = N(A) \oplus R(A)$. Since also $\sigma(A|_{R(A)} ) \subset \mathbb C_+$ we obtain that the restricted semigroup $e^{-At}|_{R(A)}$ is exponentially stable.
\end{proof}
%In particular, we have the following observation. Since the semigroup $e^{-At}|_{R(A)}$ is analytic in $X_0$ and exponentially stable, we obtain that the resolvent set $\rho(-A) $ contains a sector with opening angle $\pi/2 + \delta$ for some $\delta > 0$, that is, $\Sigma_{\pi/2 + \delta} \subset \rho(-A)$. Since we also know that the spectrum consists of isolated eigenvalues only and $\lambda = 0$ is an eigenvalue, there exists $\varepsilon > 0$ such that $\sigma(-A) \backslash \{ 0 \} \subset \mathbb C_- \backslash B_\varepsilon (0)$. In particular, there exists a small number $\kappa > 0$, such that $\sigma (-A) \subset \{  z \in \mathbb C : \operatorname{Re} z \leq - \kappa < 0 \}$.

%Therefore, it makes sense to consider the restriction of $A$ to functions $h$ with mean value zero. 
Define now a linear operator $L : D(L) \subset X_1 \pfeil \tilde X_0$ by $L(u,h) := A(u,h)$,
where
\begin{equation*}
D(L) := D(A) \cap \{ (u,h) \in X_1 : (h,1)_{L_2(\Sigma)} = 0 \},
\end{equation*}
and $\tilde X_0 := X_0 \cap \{ (u,h) \in X_0 : P_0^\Sigma h = 0 \}$. Hereby $$P_0^\Sigma h := \frac{1}{| \Sigma |} \int_\Sigma h dx.$$
Then $L$ is well-defined and $\sigma(-L) \subset \{ \lambda \in \mathbb C : \operatorname{Re} \lambda \leq -\kappa < 0 \}$ for some $\kappa > 0$, since we have a spectral gap.
\subsection{Parametrization of the nonlinear phase manifold}
We recall, cf. \eqref{930487547650487gg5}, that the transformed equations around the trivial equilibrium $(0,\Sigma) \in \mathcal E$ read as
\begin{equation} \label{9304875476504rtshh87gg5}
\begin{alignedat}{2} 
\rho \partial_t  u  - \mu \Delta u + \nabla p &= F_u(h,u,p) , &&\text{in } \Omega \backslash \Sigma, \\
\operatorname{div} u &= G_d(h,u), &&\text{in } \Omega \backslash \Sigma, \\
- \ljump \mu (Du + Du^\top) -pI \rjump e_3  &=  \sigma \Delta_{x'} h e_3 + G_S(h,u,p), && \text{on } \Sigma,   \\
\ljump u \rjump &= 0, && \text{on } \Sigma, \\
\partial_t h &= u_3 - \ljump \partial_3 \eta \rjump + G_\Sigma(h,u,\eta),\quad\quad && \text{on } \Sigma, \\
(-\nabla_{x'} h, 0)^\top \cdot \nu_{S_1} &= 0, && \text {on } \partial \Sigma, \\
\Delta \eta &= G_c(h,\eta), &&\text {in } \Omega \backslash \Sigma, \\
\eta|_{\Sigma} - \sigma \Delta_{x'} h &= G_\kappa (h), && \text {on } \Sigma, \\
\nu_{\partial\Omega} \cdot \nabla \eta|_{\partial\Omega} &= G_N(h,\eta), && \text{on } \partial\Omega \backslash\Sigma, \\
P_{S_1} \left( \mu (Du + Du^\top) \nu_{S_1} \right) &= 0, && \text{on } S_1 \backslash \partial\Sigma, \\
u \cdot \nu_{S_1} &= 0, && \text {on } S_1 \backslash \partial \Sigma, \\
u &= 0, && \text{on } S_2, \\
u(0) &= u_0, && \text{on } \Omega \backslash \Sigma, \\
h (0) &= h_0, && \text{on } \Sigma,
\end{alignedat}
\end{equation}
where $F_u(h,u,p) := a^\pm(h;D_x)(u,p) + \bar a(h,u)$, cf. \eqref{930487547650487gg5}. The nonlinear phase manifold is given by
\begin{align*}
\mathsf{PM} := \{ &(u,h) \in W^{2-2/r}_r(\Omega \backslash \Sigma) \cap B^{4-1/q-3/p}_{qp}(\Sigma) : \div u = G_d, \\  &P_\Sigma(\mu^\pm(Du+Du^\top)e_3) = ((G_S)_{1,2},0) , \; \ljump u \rjump = 0, \; (\nabla_{x'} h | n_{\partial\Sigma} ) = 0, \\ &(h|1)_{L_2(\Sigma)} = 0, \; P_{S_1} \left( \mu^\pm (Du + Du^\top) \nu_{S_1} \right) = 0, \; u \cdot \nu_{S_1} = 0, \; u|_{S_2} = 0 \}
\end{align*}
as a subset of $X_\gamma := W^{2-2/r}_r(\Omega \backslash \Sigma) \cap B^{4-1/q-3/p}_{qp}(\Sigma)$. The linear phase manifold is given by
\begin{align*}
\mathsf{PM}_0 := \{ &(u,h) \in W^{2-2/r}_r(\Omega \backslash \Sigma) \cap B^{4-1/q-3/p}_{qp}(\Sigma) : \div u = 0, \\  &P_\Sigma(\mu^\pm(Du+Du^\top)e_3) = 0, \; \ljump u \rjump = 0, \; (\nabla_{x'} h | n_{\partial\Sigma} ) = 0, \\ &(h|1)_{L_2(\Sigma)} = 0, \; P_{S_1} \left( \mu^\pm (Du + Du^\top) \nu_{S_1} \right) = 0, \; u \cdot \nu_{S_1} = 0, \; u|_{S_2} = 0 \}.
\end{align*}
We now refer to Section 4.2 in \cite{wilkehabil}, where it is shown that there is a local parametrization of $\mathsf{PM}$ over $\mathsf{PM}_0$ around zero. More precisely there is a small $r > 0$, such that for every $(u_0,h_0) \in B(r,0) \subset \mathsf{PM}$ there is a $C^2$-function $\phi$ and a decomposition
\begin{equation} \label{uuu90q3846708w76bfg}
(u_0,h_0) = (\tilde u_0, \tilde h_0) + (\phi(\tilde u_0,\tilde h_0) , 0 ) , \quad (\tilde u_0, \tilde h_0) \in \mathsf{PM}_0.
\end{equation}
For details we refer to Proposition 4.3 and Section 4.2 in \cite{wilkehabil}.
\subsection{Convergence to equilibria}
We now state and prove the main result.
\begin{theorem}
The trivial equilibrium $(0,\Sigma) \in \mathcal E$ is stable in the following sense. For each $\varepsilon >0$ there exists some $\delta = \delta(\varepsilon) >0$ such that for all initial values $(u_0,h_0) \in X_\gamma \cap \mathsf{PM}$ satisfying
\begin{equation} \label{90347650387465084765gg}
| u_0 |_{W^{2-2/r}_r(\Omega \backslash \Sigma)} + |h_0|_{B^{4-1/q-3/p}_{qp}(\Sigma)} \leq \delta(\varepsilon),
\end{equation}
there exists some global in time solution
\begin{gather*}
u \in W^1_r(\R_+;L_r(\Omega)) \cap L_r(\R_+; W^2_r(\Omega \backslash \Sigma)), \\ h \in W^1_p(\R_+; W^{1-1/q}_q(\Sigma)) \cap L_p(\R_+;W^{4-1/q}_q(\Sigma)),
\end{gather*}
such that 
\begin{equation*}
| u(t) |_{W^{2-2/r}_r(\Omega \backslash \Sigma)} + |h(t)|_{B^{4-1/q-3/p}_{qp}(\Sigma)} \leq \varepsilon, \quad t \in \R_+.
\end{equation*}
Moreover, 
\begin{equation*}
| u(t) |_{W^{2-2/r}_r(\Omega \backslash \Sigma)} + |h(t)-P_0^\Sigma h_0|_{B^{4-1/q-3/p}_{qp}(\Sigma)} \pfeil_{t \pfeil \infty} 0,
\end{equation*}
where $P_0^\Sigma h_0 = \frac{1}{|\Sigma|}\int_\Sigma h_0$ is the mean value of $h_0$.
The convergence is at an exponential rate.
\end{theorem}
\begin{proof} We follow the lines of \cite{abelswilke} and \cite{wilkehabil}.
Let $\varepsilon >0$ be given and $(u_0,h_0) \in X_\gamma \cap \mathsf{PM}$ such that the smallness condition \eqref{90347650387465084765gg} holds for some $\delta > 0$ to be specified later. By \eqref{uuu90q3846708w76bfg}, we can decompose the initial data
\begin{equation*}
(u_0,h_0) =(0, P_0^\Sigma h_0) +  (\tilde u_0, \tilde h_0) + (\phi(\tilde u_0,\tilde h_0) , 0 ) ,
\end{equation*}
where $(\tilde u_0, \tilde h_0) + (\phi(\tilde u_0,\tilde h_0) , 0 ) \in \mathsf{PM}$ and $ (\tilde u_0, \tilde h_0) \in \mathsf{PM}_0$. 
%Note that we need to subtract the mean value of $h_0$ first to obtain that $(u_0, h_0 - P_0^\Sigma h_0) \in \mathsf{PM}$.
We now want to decompose the solution $(u(t),h(t))$ suitably and write
\begin{equation*}
(u(t),h(t)) = (0, P_0^\Sigma h_0) + ( \tilde u(t), \tilde h(t)) + ( \bar u (t), \bar h(t) ), \quad t \in \R_+,
\end{equation*}
where $(\tilde u(t), \tilde h(t)) \in \mathsf{PM}_0$ for $t \in \R_+$, and estimate each term separately. We consider the two coupled systems
\begin{equation} \label{eetzuetrthhzu87gg5} 
\begin{alignedat}{2} 
 \omega \rho \bar u + \rho \partial_t  \bar u  - \mu^\pm \Delta \bar u + \nabla \bar \pi &= F_u(P_0^\Sigma h_0 +\tilde h + \bar h,\tilde u + \bar u,\tilde \pi + \bar \pi) ,  &\text{in } \Omega \backslash \Sigma, \\
\operatorname{div} \bar u &= G_d(P_0^\Sigma h_0 +\tilde h + \bar h,\tilde u + \bar u), &\text{in } \Omega \backslash \Sigma, \\
-P_\Sigma ( \ljump \mu^\pm (D \bar u+ D \bar u^\top) e_3 \rjump ) &= G^\parallel_S (P_0^\Sigma h_0 +\tilde h + \bar h,\tilde u + \bar u), &\text{on } \Sigma, \\
- 2 \ljump \mu^\pm \partial_3 \bar u_3 \rjump + \ljump \bar \pi \rjump - \sigma \Delta_{x'} \bar h &=   G_S^\perp (P_0^\Sigma h_0 +\tilde h + \bar h,\tilde u + \bar u), &\text{on } \Sigma,  \\
\ljump \bar  u \rjump &= 0, & \text{on } \Sigma, \\
\omega \bar h + \partial_t \bar h -\bar u_3 + \ljump \partial_3 \bar \eta \rjump &=   G_\Sigma(P_0^\Sigma h_0 + \tilde h + \bar h, \tilde u + \bar u, \tilde \eta + \bar \eta), & \text{on } \Sigma, \\
(-\nabla_{x'} \bar h, 0)^\top \cdot \nu_{S_1} &= 0, & \text {on } \partial \Sigma, \\
\Delta \bar \eta &= G_c(P_0^\Sigma h_0 +\tilde h + \bar h,\tilde\eta+\bar \eta), & \text{in }\Omega \backslash \Sigma, \\
\bar \eta|_{\Sigma} - \sigma \Delta_{x'} \bar h &= G_\kappa (P_0^\Sigma h_0 +\tilde h + \bar h), & \text{on } \Sigma, \\
\nu_{\partial\Omega} \cdot \nabla \bar \eta|_{\partial\Omega} &= G_N(P_0^\Sigma h_0 +\tilde h + \bar h,\tilde \eta + \bar \eta), & \text{on } \partial\Omega \backslash\Sigma, \\
P_{S_1} \left( \mu^\pm (D\bar u + D\bar u^\top) \nu_{S_1} \right) &= 0, &\text{on } S_1 \backslash \partial\Sigma, \\
\bar u \cdot \nu_{S_1} &= 0, & \text{on } S_1 \backslash \partial \Sigma, \\
\bar u &= 0, & \text{on } S_2, \\
\bar u(0) &= \phi(\tilde u_0, \tilde  h_0), & \text{on } \Omega \backslash \Sigma, \\
\bar h (0) &= 0, & \text{on } \Sigma, 
\end{alignedat}
\end{equation}
where $\omega >0$, and
\begin{equation} \label{eetzuetrththzu87gg5}
\begin{alignedat}{2} 
  \rho \partial_t  \tilde u  - \mu^\pm \Delta \tilde u + \nabla \tilde \pi &= \omega \rho (I-T_1)\bar u , &\text{in } \Omega \backslash \Sigma, \\
\operatorname{div} \tilde u &= 0, &\text{in } \Omega \backslash \Sigma, \\
-P_\Sigma ( \ljump \mu^\pm (D \tilde u+ D \tilde u^\top) e_3 \rjump ) &= 0, &\text{on } \Sigma, \\
- 2 \ljump \mu^\pm \partial_3 u_3 \rjump + \ljump \tilde \pi \rjump - \sigma \Delta_{x'} \tilde h &=   0, &\text{on } \Sigma,  \\
\ljump \tilde  u \rjump &= 0, & \text{on } \Sigma, \\
 \partial_t \tilde h -\tilde u_3 + \ljump \partial_3 \tilde \eta \rjump &=   \omega ( \bar h - P_0^\Sigma \bar h), \qquad\qquad & \text{on } \Sigma,  \\
(-\nabla_{x'} \tilde h, 0)^\top \cdot \nu_{S_1} &= 0, & \text {on } \partial \Sigma, \\
\Delta \tilde \eta &= 0, &\text {in } \Omega \backslash \Sigma, \\
\tilde \eta|_{\Sigma} - \sigma \Delta_{x'} \tilde h &= 0, & \text {on } \Sigma, \\
\nu_{\partial\Omega} \cdot \nabla \tilde \eta|_{\partial\Omega} &= 0, & \text{on } \partial\Omega \backslash\Sigma, \\
P_{S_1} \left( \mu^\pm (D\tilde u + D\tilde u^\top) \nu_{S_1} \right) &= 0, & \text{on } S_1 \backslash \partial\Sigma, \\
\tilde u \cdot \nu_{S_1} &= 0, & \text {on } S_1 \backslash \partial \Sigma, \\
\tilde u &= 0, & \text{on } S_2, \\
\tilde u(0) &= \tilde u_0, & \text{on } \Omega \backslash \Sigma, \\
\tilde h (0) &= \tilde h_0 , & \text{on } \Sigma. 
\end{alignedat}
\end{equation}
Let us note a few things here. The right hand side of $\eqref{eetzuetrththzu87gg5}_1$ can equivalently be written as $\omega \rho (I-T_1) \bar u = \omega \rho \bar u - \omega \rho \nabla \bar q$, where $\bar q \in \dot H^1_r(\Omega \backslash \Sigma)$ is the unique solution of the weak transmission problem
\begin{align*}
(\nabla \bar q | \nabla \psi)_{L_2(\Omega)} &= ( \bar u  | \nabla \psi )_{L_2(\Omega)}, \quad && \text{for all } \psi \in W^1_{r'}(\Omega), \\
\ljump \bar q \rjump &= 0,& & \text{on } \Sigma.
\end{align*}
%This will be needed since $\nabla \bar \pi$ and $\nabla \tilde \pi$ solve 
%\begin{align}
%(\nabla \bar \pi | \nabla \psi)_{L_2(\Omega)} &= ( -\mu^\pm \Delta \bar u - \omega \bar u + F_u(P_0^\Sigma h_0 +\tilde h + \bar h,\tilde u + \bar u,\tilde \pi + \bar \pi) | \nabla \psi )_{L_2(\Omega)},&& \\ & &   \text{for all } \psi \in W^1_{r'}(\Omega),&  \\
%\ljump \bar \pi \rjump &= 2\ljump \mu^\pm \partial_3 \bar u_3 \rjump + \sigma \Delta_{x'} \bar h + G_S^\perp (P_0^\Sigma h_0 +\tilde h + \bar h,\tilde u + \bar u) , \text{on } \Sigma,&
%\end{align}
%and
%\begin{subequations}
%\begin{align}
%(\nabla \tilde \pi | \nabla \psi)_{L_2(\Omega)} &= ( \mu^\pm \Delta \tilde u  | \nabla \psi )_{L_2(\Omega)}, \quad & \text{for all } \psi \in W^1_{r'}(\Omega), \\
%\ljump \tilde \pi \rjump &= 0, & \text{on } \Sigma.
%\end{align}
%\end{subequations}
Furthermore, the initial value $\tilde h_0$ in \eqref{eetzuetrththzu87gg5} is mean value free.
% since $h_0 = P_0^\Sigma h_0 + \tilde h_0$, hence taking the projection $P_0^\Sigma$ yields $P_0^\Sigma \tilde h_0 = 0$.
Note that the right hand side of $\eqref{eetzuetrththzu87gg5}_6$ is mean value free as well, hence an integration of $\eqref{eetzuetrththzu87gg5}_6$ over $\Sigma$ yields that $\tilde h$ stays mean value free for all times $t >0$. In particular, we can equivalently rewrite \eqref{eetzuetrththzu87gg5} in the projected base space $\tilde X_0$ as
\begin{equation} \label{9384056348076gdf}
\frac{d}{dt} \tilde z (t)+ L\tilde z(t) = R(\bar z)(t), \quad t >0 , \quad z(0) = \tilde z_0 := (\tilde u_0, \tilde h_0).
\end{equation}
Here, $\tilde z := (\tilde u, \tilde h)$, $\bar z := (\bar u, \bar h)$ and $R(\bar z) := (\omega (I-T_1)\bar u, (I-P_0^\Sigma) \bar h )$.
Note that by Lemma \ref{39840567349576340576834065734}, the spectral bound of $-L$ satisfies $s(-L) \leq - \kappa < 0$ and the restricted semigroup $e^{-Lt}$ is exponentially stable on $\tilde X_0$. 
%By the variation of constants formula, we can solve problem \eqref{9384056348076gdf} on bounded intervals $J = (0,T)$, $0 < T < \infty$, by
%\begin{equation}
%\tilde  z(t) = e^{-Lt} \tilde z_0 + \int_0^t e^{-L(t-s)} R(\bar z) (s) ds, \quad t \in J.
%\end{equation}
%Hence the formula holds also on $\R_+$, supposed that $R(\bar z) \in L_r(\R_+; W^{2-2/r}_r(\Omega \backslash \Sigma)) \times L_p(\R_+; W^{1-1/q}_q(\Sigma))$.

We now solve this evolution equation in exponentially time-weighted spaces to get suitable decay estimates, cf. \cite{abelswilke} and \cite{wilkehabil}.
Let us introduce notation. Let
$
\mathbb E_u (\R_+) := H^1_r(\R_+; L_r (\Omega)) \cap L_r(\R_+ ; H^2_r(\Omega \backslash \Sigma))$ and 
$$
\mathbb E_h (\R_+) := W^1_p(\R_+; W^{1-1/q}_q(\Sigma)) \cap L_p(\R_+; W^{4-1/q}_q(\Sigma)).
$$
For $\beta \in [0, -s(-L) )$ define
\begin{gather*}
 e^{-\beta t}\mathbb E_u (\R_+)  := \{  w \in L_r(\R_+ ; L_r(\Omega)) : e^{\beta t} w \in \mathbb E_u (\R_+) \}, \\
e^{-\beta t}\mathbb E_h (\R_+)   := \{  w \in L_p(\R_+ ; L_q(\Omega)) : e^{\beta t} w \in \mathbb E_h (\R_+) \}.
\end{gather*}
In a similar way we define $e^{-\beta t} L_r(\R_+;L_r(\Omega))$.
% we say $w \in e^{-\beta t} L_r(\R_+;L_r(\Omega))$, if and only if $e^{\beta t} w \in L_r(\R_+;L_r(\Omega))$.
Since $0 \leq \beta < - s(-L)$, we obtain that for every 
\begin{equation*}
(f_u,f_h) \in e^{-\beta t}[ L_r(\R_+;L_r(\Omega)) \times L_p(\R_+;W^{1-1/q}_q(\Sigma)) ],
\end{equation*}
and $(\hat u_0, \hat h_0) \in X_\gamma$ 
there is a unique solution
$
(u,h) \in e^{-\beta t}[ \mathbb E_u (\R_+) \times \mathbb E_h (\R_+) ]
$
of the linear evolution problem 
\begin{equation*}
\partial_t (u,h) + L(u,h) = (f_u,f_h), \quad t \in \R_+, \quad (u,h)|_{t=0} = ( \hat u_0, \hat h_0),
\end{equation*}
by maximal regularity in exponentially time-weighted spaces.
Furthermore, there is some $M>0$ such that
\begin{equation*}
\begin{alignedat}{1}
| (u,h) &|_{ e^{-\beta t}[ \mathbb E_u (\R_+) \times \mathbb E_h (\R_+) ] } \\ &\leq M | (f_u,f_h,\hat u_0, \hat h_0 ) |_{ e^{-\beta t}[ L_r(\R_+;L_r(\Omega)) \times L_p(\R_+;W^{1-1/q}_q(\Sigma)) ] \times X_\gamma }.
\end{alignedat}
\end{equation*}
In particular, we may then easily solve \eqref{9384056348076gdf} in dependence of $\bar z = (\bar u, \bar h)$,
\begin{equation} \label{09376450837465078fdg}
(\tilde u, \tilde h ) = \left( \frac{d}{dt} + L , \operatorname{tr}|_{t=0} \right)^{-1} ( \omega (I-T_1)\bar u, (I-P_0^\Sigma)\bar h, \tilde u_0, \tilde h_0).
\end{equation}
Let us now discuss problem \eqref{eetzuetrthhzu87gg5}. For given $\omega > 0$, let $L_\omega$ be given by the left hand side of \eqref{eetzuetrthhzu87gg5} and $N$ the collection of nonlinearities on the right hand side. Then we can rewrite problem \eqref{eetzuetrthhzu87gg5} in the shorter form
\begin{equation*}
L_\omega \bar w = N(w_\infty + \tilde w + \bar w), \quad (\bar u, \bar h)(0) = ( \phi( \tilde u_0,\tilde h_0) , 0 ),
\end{equation*}
where $\bar w := (\bar u, \bar h , \bar \pi, \bar \eta)$, $\tilde w := (\tilde u, \tilde h, \tilde \pi, \tilde \eta)$ and $w_\infty := (0, P_0^\Sigma h_0 , 0 ,0)$. Note at this point that $w_\infty$ is constant and $N$ does not explicitly depend on $w_\infty$. Furthermore, due to the first part of the proof, $\tilde w$ depends only on $(\tilde u_0, \tilde h_0 , \bar u , \bar h)$, cf. \eqref{09376450837465078fdg}.

In order to solve problem \eqref{eetzuetrthhzu87gg5} we need to resolve the initial data and the compatibility conditions at $t=0$ properly. By solving certain auxiliary problems in exponentially weighted spaces, we may construct an extension operator
\begin{equation*}
\operatorname{ext}_\beta : \bar X_\gamma \pfeil e^{-\beta} [ \mathbb E_u(\R_+) \times \mathbb E_h(\R_+) ],
\end{equation*}
satisfying $\operatorname{ext}_\beta(v,g)|_{t=0} = (v,g)$ for all $(v,g) \in \bar X_\gamma$, where
\begin{align*}
\bar X_\gamma := \{ (u,h) \in X_\gamma : & u|_{S_2} = 0, \; (u|\nu_{S_1}) = 0,\; P_{S_1}(\mu^\pm(Du+Du^\top)\nu_{S_1}) = 0, \\ &  \ljump u \rjump = 0, \; (\nabla_{x'} h | \nu_{\partial\Sigma} ) = 0 \},
\end{align*}
cf. \cite{wilkehabil}. Now define
\begin{equation*}
M(\tilde u_0, \tilde h_0, \bar w) := N(w_\infty + \tilde w + \bar w + \operatorname{ext}_\beta[ (\phi(\tilde u_0,\tilde h_0) , 0 ) - (\bar u(0),\bar h(0)) ] ).
\end{equation*}
By construction,
$
M(\tilde u_0, \tilde h_0, \bar w)|_{t= 0} = 
N( u_0, h_0 , 0 ,0 ).
$
This allows us to solve the problem
\begin{equation*}
L_\omega \bar w = M(\tilde u_0, \tilde h_0, \bar w), \quad (\bar w_1, \bar w_2)|_{t=0} = (\phi(\tilde u_0, \tilde h_0),0),
\end{equation*}
by the implicit function theorem, since all relevant compatibility conditions at $t=0$ are satisfied. Following the lines of \cite{abelswilke}, we obtain that there is some small $\rho > 0$ and a ball $B(0,\rho) \subset X_\gamma \cap \mathsf{PM}_0$, such that there is a $\Phi \in C^1(B(0,\rho) ; e^{-\beta t} [ \mathbb E_u(\R_+) \times \mathbb E_h(\R_+) \times \mathbb E_\pi (\R_+) \times \mathbb E_\eta (\R_+) ] )$ satisfying $\bar w = \Phi(\tilde u_0, \tilde h_0)$. By construction, $\bar w$ is the solution of \eqref{eetzuetrthhzu87gg5}. Here, $\mathbb E_\pi (\R_+ ) := L_r(\R_+; \dot H^1_r(\Omega \backslash \Sigma))$, $\mathbb E_\eta (\R_+) := L_p(\R_+; W^2_q(\Omega \backslash \Sigma))$.

We then obtain that the convergence $(u(t),h(t)) \pfeil (0, P_0^\Sigma h_0)$ in $X_\gamma$ is at an exponential rate. The proof is complete.
\end{proof}

\section*{Acknowledgements} M.R. would like to thank Harald Garcke for pointing out existing work in \cite{aggthermo} and Helmut Abels for inspiring discussions regarding the boundary conditions. The work of M.R. is financially supported by the DFG graduate school GRK 1692. The support is gratefully acknowledged.
\bibliographystyle{plain}
\bibliography{bibo}
\end{document}